\documentclass[11pt,letterpaper]{amsart}

\usepackage[utf8]{inputenc} 
\usepackage[T1]{fontenc}    
\usepackage{lmodern}
\usepackage{hyperref}       
\usepackage{url}            
\usepackage{microtype}      
\usepackage{amsfonts}
\usepackage{amsthm}
\usepackage{amsmath}
\usepackage{amssymb}
\usepackage{todonotes}

\usepackage[capitalise]{cleveref}
\usepackage{bm}
\usepackage[toc,page]{appendix}
\usepackage{graphicx}
\usepackage{comment}
\usepackage{enumitem}
\usepackage{listings}
\usepackage{bbold}
\usepackage{enumitem}

\usepackage{soul}

\allowdisplaybreaks
\setlist{nosep}

\newtheorem{theorem}{Theorem}[section]
\newtheorem*{theoremst}{Theorem}
\newtheorem{corollary}[theorem]{Corollary}
\newtheorem{lemma}[theorem]{Lemma}

\newtheorem{proposition}[theorem]{Proposition}

\theoremstyle{definition}
\newtheorem{definition}{Definition}[section]
 
\theoremstyle{remark}
\newtheorem{remark}{Remark}[section]

\newcommand{\sca}[2]{\langle #1 | #2\rangle}
\newcommand{\nr}[1]{\left\lVert #1\right\rVert}
\newcommand{\abs}[1]{\left\vert #1\right\vert}
\newcommand{\Nsp}{\mathbb{N}}
\newcommand{\Rsp}{\mathbb{R}}
\newcommand{\id}{\mathrm{id}}
\newcommand{\DD}{\mathrm{D}}
\newcommand{\Class}{\mathcal{C}}
\newcommand{\Haus}{\mathcal{H}}

\newcommand{\Wass}{\mathrm{W}}

\newcommand{\X}{\mathcal{X}}
\newcommand{\Y}{\mathcal{Y}}
\newcommand{\restr}[1]{\left.#1\right|}
\DeclareMathOperator{\diam}{diam}
\renewcommand{\d}{\mathrm{d}}
\newcommand{\dd}{\mathrm{d}}

\renewcommand{\L}{\mathrm{L}}

\newcommand{\Lip}{\mathrm{Lip}}
\newcommand{\eps}{\varepsilon}
\renewcommand{\epsilon}{\varepsilon}

\newcommand{\Prob}{\mathcal{P}}
\newcommand{\vol}{\mathrm{vol}}

\newcommand{\Sph}{\mathcal{S}}

\DeclareMathOperator{\spt}{spt}
\DeclareMathOperator{\Var}{Var}

\newcommand{\Esp}{\mathbb{E}}

\newcommand{\Kant}{\mathcal{K}}

\newcommand{\mphi}{m_\phi}
\newcommand{\Mphi}{M_\phi}
\newcommand{\mrho}{m_\rho}
\newcommand{\Mrho}{M_\rho}

\begin{document}

\title[Quantitative stability of optimal transport maps]{Quantitative stability of optimal transport maps under variations of the target measure}
\author{Alex Delalande}
\address{Laboratoire de Mathématiques d’Orsay, Univ. Paris-Sud, CNRS, Université Paris-Saclay, 91405 Orsay, France \and Inria Saclay, Ile-de-France}
\author{Quentin Mérigot}
\address{Laboratoire de Mathématiques d’Orsay, Univ. Paris-Sud, CNRS, Université Paris-Saclay, 91405 Orsay, France \and Institut universitaire de France (IUF)}

\maketitle

\begin{abstract}
   This work studies the quantitative stability of the quadratic optimal transport map between a fixed probability density $\rho$ and  a probability measure $\mu$ on $\Rsp^d$, which we denote $T_\mu$. Assuming that the source density $\rho$ is bounded from above and below on a compact convex set, we prove that the map $\mu\mapsto T_\mu$ is bi-Hölder continuous on large families of probability measures, such as the set of probability measures whose moment of order $p>d$ is bounded by some constant. 
   These stability estimates show that the \emph{linearized optimal transport} metric $\Wass_{2,\rho}(\mu,\nu) = \nr{T_\mu - T_\nu}_{\L^2(\rho,\Rsp^d)}$ is bi-Hölder equivalent to the $2$-Wasserstein distance on such sets, justifiying its use in applications.
\end{abstract}

\section{Introduction}

Let $\Prob_2(\Rsp^d)$ be the set of probability measures with finite second moment over $\Rsp^d$ and $\rho, \mu \in \Prob_2(\Rsp^d)$.
The optimal transport problem between $\rho$ and $\mu$ with respect to the quadratic 
cost $c(x,y) = \nr{x-y}^2$ is the following minimization problem, where the minimum is taken over 
the set $\Pi(\rho, \mu)$ of \emph{transport plans} between $\rho$ and $\mu$, that is the set of probability measures over $\Rsp^d \times \Rsp^d$ with marginals $\rho$ and $\mu$:
\begin{align*}
\min_{\gamma \in \Pi(\rho, \mu)} \int_{\Rsp^d \times \Rsp^d} \nr{x - y}^2 \dd \gamma(x, y).
\end{align*}
The square root of the value of this problem is called the \emph{$2$-Wasserstein distance} between $\rho$ and $\mu$ and is denoted $\Wass_2(\rho,\mu)$. 
A theorem of Brenier \cite{Brenier} asserts that if $\rho$ is absolutely continuous with respect to the Lebesgue measure, the minimizer of the optimal transport problem is unique, and is induced by a map $T = \nabla \phi$, where $\phi$ is a convex function that verifies $\nabla \phi_\# \rho = \mu$. We recall that $T_\# \rho$ denotes the image measure of $\rho$ under the map $T$. In our precise setting, where 
the density $\rho$ is bounded from above and below on a compact convex set, the potential $\phi$ is uniquely defined in $\L^2(\rho)$ up to an additional constant. Square-summability of $\phi$ follows from the Poincaré-Wirtinger inequality on $\X$.

\begin{definition}[Potentials and maps]\label{def:Brenier}
We fix a probability measure $\rho \in \Prob_2(\Rsp^d)$, which we assume to be absolutely continuous with respect to the Lebesgue measure and supported over a
compact convex set $\X$. We assume that the density of $\rho$ is bounded from above and below by positive constants on $\X$. 
Given $\mu\in\Prob_2(\Rsp^d),$ we call
\begin{itemize}
    \item \emph{Brenier map} and denote $T_\mu$ the (unique) optimal transport map between $\rho$ and $\mu$;
\item \emph{Brenier potential} the unique lower semi-continuous convex function $\phi_\mu \in \L^2(\rho)$ such that $T_\mu  = \nabla \phi_\mu$ and which satisfies 
$\int_{\X} \phi_\mu\dd\rho = 0$; 
\item \emph{dual potential} the convex conjugate of $\phi_\mu$, denoted $\psi_\mu$:
$$\forall y\in \Rsp^d, \quad \psi_\mu(y) = \max_{x\in\X}\sca{x}{y} - \phi_\mu(x), $$
where the maximum is attained by lower semi-continuity of the convex function $\phi_\mu$ on the compact convex set $\X$. 
\end{itemize}
\end{definition}

Since $\mu$ is the image of $\rho$ under $T_\mu$, the mapping  $\mu \in (\Prob_2(\Rsp^d), \Wass_2) \mapsto T_\mu \in\L^2(\rho,\Rsp^d)$ is obviously injective. Using that 
$(T_\mu,T_\nu)_\#\rho$ is a coupling between $\mu$ and $\nu$, one can actually prove that this mapping increases distances, namely
$$ \forall \mu,\nu\in\Prob_2(\Rsp^d),\quad \Wass_2(\mu, \nu) \leq \nr{T_\mu - T_\nu}_{\L^2(\rho, \Rsp^d)}. $$
This mapping is also continuous: if a sequence of probability measures $(\mu_n)_n$ converges to some $\mu$ in $(\Prob_2(\Rsp^d), \Wass_2)$, then $T_{\mu_n}$ converges to $T_\mu$ in $\L^2(\rho, \Rsp^d)$. This continuity property is for instance implied by Corollary~5.23 in \cite{villani2008optimal}, together with the dominated convergence theorem. However, we note that the arguments used to prove this general continuity result are non-quantitative.

\subsection*{Linearized Optimal Transport}
These two properties of the map $\mu\mapsto T_\mu$ motivated its use to embed the metric space $(\Prob_2(\Rsp^d), \Wass_2)$ into the Hilbert space $\L^2(\rho, \Rsp^d)$ \cite{LOT_ref_image}.

This approach is often referred to as the \emph{Linearized Optimal Transport} (LOT) framework and has shown great results in applications to image processing: 
\begin{itemize}
    \item \cite{LOT_ref_image, LOT5, LOT4, LOT6} used this idea to perform pattern recognition in images for various tasks, including discrimination of nuclear chromatin patterns in cancer cells, detection of differences in facial expressions, bird species, galaxy morphologies, sub-cellular protein distributions, detection and visualization of cell phenotype differences from microscopy images, or finally jets tagging of collider data in collider physics. 
    \item \cite{LOT3}  considered this framework for generative modelling of images, with experiments showcasing the generative modelling of digits and faces images, PET scans in the context of Alzheimer’s disease neuroimaging, or thyroid nuclei images. 
    \item \cite{LOT7}  followed this approach for improving the resolution of faces images.
\end{itemize}
At this stage, the good practical behavior of the linearized optimal transport framework is not justified from a mathematical viewpoint. A practical benefit of the embedding is to enable the use of the classical Hilbertian statistical toolbox on families of probability measures while keeping some features of the Wasserstein geometry. A particularly nice feature of the embedding $\mu\mapsto T_\mu$ is that its image in $\L^2(\rho,\Rsp^d)$ is convex, i.e. barycenters of optimal transport maps are optimal transport maps. Working with this embedding is equivalent to replacing the Wasserstein distance by the distance 
$$ \Wass_{2,\rho}(\mu,\nu) = \nr{T_\mu - T_\nu}_{\L^2(\rho,\Rsp^d)}. $$
We note that the geodesic curves with respect to the distance $\Wass_{2,\rho}$ are called the \emph{generalized geodesics} in the book of Ambrosio, Gigli, Savaré \cite{ambrosio2008gradient}. The choice of the Brenier map between a reference measure $\rho$ and a measure $\mu$ as an embedding of $\mu$ may also be motivated by the Riemannian interpretation of the Wasserstein geometry \cite{otto2001geometry,ambrosio2008gradient}. 
In this interpretation, the tangent space to $\Prob_2(\Rsp^d)$ at $\rho$ is included in $\L^2(\rho,\Rsp^d)$. The Brenier map minus the identity, $T_{\mu}-\id$, can be regarded as the vector in the tangent space at $\rho$ which supports the Wasserstein geodesic from $\rho$ to $\mu$. In the Riemannian language again, the map $\mu\mapsto T_\mu - \id$ would be called a \emph{logarithm}, i.e. the inverse of the Riemannian exponential map: it sends a probability measure $\mu$ in the (curved) manifold $\Prob_2(\Rsp^d)$ to a vector $T_{\mu}-\id$ belonging to  the linear space $\L^2(\rho,\Rsp^d)$.
This establishes a connection between the linearized optimal transport framework idea and similar strategies used to extend statistical inference notions such as principal component analysis  to manifold-valued data, e.g. \cite{fletcher2004principal,stat2}. 

 It is quite natural to expect that the embedding $\mu\mapsto T_\mu$  retains some of the geometry of the underlying space, or equivalently that the metric $\Wass_{2,\rho}$ is comparable, in some coarse sense, to the Wasserstein distance. The main difficulty, which we study in this article, is to establish quantitative (e.g. Hölder) continuity properties for the mappings $\mu \mapsto T_\mu$ and $\mu\mapsto \phi_\mu$. We note that such stability estimates are also important in numerical analysis and in statistics, where a probability measure of interest $\mu \in \Prob_2(\Rsp^d)$ is often approximated by a sequence of finitely supported measures $(\mu_n)_n$: convergence rates of quantities related to the sequence $(T_{\mu_n})_n$ toward a quantity related to $T_\mu$ may then be directly deduced from quantitative stability estimates controlling $\nr{T_{\mu_n} - T_\mu}_{\L^2(\rho, \Rsp^d)}$ with $\Wass_2(\mu_n, \mu)$.

\subsection*{Existing results} We focus here on the already known stability results on the mapping $\mu \mapsto T_\mu$, starting with negative results.
We first note that explicit examples show that the mapping $\mu \mapsto T_\mu$ is in general not better than $\frac{1}{2}$-Hölder, see \S4 in \cite{gigli2011holder} or Lemma 5.1 in \cite{pmlr-v108-merigot20a}. A much stronger negative result comes from Andoni, Naor and Neiman \cite[Theorem 7]{andoni2018snowflake} showing that 
one cannot construct a bi-Hölder embedding of $(\Prob_2(\Rsp^d), \Wass_2),$ $d\geq 3$, into a Hilbert space:
\begin{theoremst}[Andoni, Naor, Neiman]
$(\Prob_2(\Rsp^3), \Wass_2)$ does not admit a uniform, coarse or quasisymmetric embedding into any Banach space of nontrivial type.
\end{theoremst}
This theorem implies in particular that one cannot hope to prove that $\mu\mapsto T_\mu$ is bi-Hölder on the whole set $\Prob_2(\Rsp^d)$ of probability measures with finite second moment. 

Existing quantitative stability results can be summed up under the two following statements. A first result due to Ambrosio and reported in \cite{gigli2011holder}, shows a local $1/2$-Hölder behaviour near probability densities $\mu$ whose associated Brenier map $T_\mu$ is Lipschitz continuous. We quote here a variant of this statement, from \cite{pmlr-v108-merigot20a}:
\begin{theoremst}[Ambrosio] 
Let $\rho$ be a probability density over a compact set $\X$. Let $\Y \subset \Rsp^d$ be a compact set and $\mu,\nu \in \Prob(\Y)$. Assume that the Brenier map $T_\mu$ from $\rho$ to $\mu$ is $L$-Lipschitz. Then,
\begin{equation*}
  \|T_\mu - T_\nu\|_{\L^2(\rho, \Rsp^d)} \leq 2 \sqrt{\diam(\X) L} \Wass_1(\mu, \nu)^{1/2}.
\end{equation*}
\end{theoremst}
Assuming Lipschitzness of the Brenier map is rather strong. First, it implies that the support of $\mu$ is connected, so that the previous theorem cannot be applied when 
both $\mu$ and $\nu$ are finitely supported. In addition, to prove that $T_\mu$ is Lipschitz one has to invoke the regularity theory for optimal transport maps, which requires very strong assumptions on $\mu$, in particular that its support $\spt(\mu)$ is convex. A more recent result, due to Berman \cite{Berman2020}, proves quantitative stability of the map $\mu\mapsto T_\mu$ under milder assumptions on the target probability measures. Berman proves a stability result on the inverse transport maps when the target measure is bound to remain in a fixed compact set 
 \cite[Proposition 3.2]{Berman2020}. This result implies quantitative stability of the Brenier maps; we refer to Corollary 2.4 in \cite{pmlr-v108-merigot20a} for a precise statement.
\begin{theoremst}[Berman]
Let $\rho$ be a probability density over a  compact convex set $\X$, bounded from above and below by positive constants. Let $\Y$ be a bounded connected open subset of $\Rsp^d$ with a Lipschitz boundary. Then there exists 
a constant $C$ depending only on $\rho$, $\X$ and $\Y$ such that for any $\mu,\nu \in \Prob(\Y)$,
\begin{equation*}
  \|T_\mu - T_\nu\|_{\L^2(\rho, \Rsp^d)} \leq C \Wass_1(\mu, \nu)^{\frac{1}{2^{(d-1)}(d+2)}}.
\end{equation*}
\end{theoremst}
Unlike in Ambrosio's theorem, the Hölder behavior given does not depend on the regularity of the transport map $T_\mu$. On the other hand, the Hölder exponent 
depends exponentially on the ambient dimension $d$. As we will see below, this is not optimal.

\subsection*{Contributions} 
In this article, we prove quantitative stability results for quadratic optimal transport maps between a probability density $\rho$ and 
target measure $\mu$. We do not assume that $\mu$ is compactly supported.

Introducing $M_p(\mu) = \int_{\Rsp^d} \nr{x}^p \dd \mu(x)$ the $p$-th moment of $\mu \in \Prob_2(\Rsp^d)$, we prove in particular the following theorem. 
We denote by $C_{a_1,\hdots,a_n}$ a non-negative constant which depends on $a_1,\hdots,a_n$.

\begin{theoremst}[\cref{cor:stab-pot-mp,th:stability-ot-maps,th:stability-ot-maps-compact-case}] 
Let $\X$ be a compact convex set and let $\rho$ be a probability density on $\X$, bounded from above and below by positive constants.
Let $p>d$ and $p\geq 4$. Assume that $\mu, \nu \in \Prob_2(\Rsp^d)$ have bounded $p$-th moment, i.e. $\max(M_p(\mu), M_p(\nu)) \leq M_p < +\infty$. Then 
$$ \nr{T_\mu - T_\nu}_{\L^2(\rho, \Rsp^d)} \leq C_{d, p, \X, \rho, M_p} \Wass_1(\mu, \nu)^{\frac{p}{6p + 16d}},$$
$$ \nr{\phi_\mu - \phi_\nu}_{\L^2(\rho)} \leq C_{d,p,\X,\rho,M_p} \Wass_1(\mu,\nu)^{1/2}.$$
If $\mu,\nu$ are supported on a compact set $\Y$, we have an improved Hölder exponent for the Brenier map:
$$ \nr{T_\mu - T_\nu}_{\L^2(\rho, \Rsp^d)} \leq C_{d, \X, \Y, \rho} \Wass_1(\mu, \nu)^{\frac{1}{6}}.$$
\end{theoremst}

\begin{remark}[Comparison between $\Wass_1$ and $\Wass_2$] We note that since $\Wass_1\leq \Wass_2$, the estimates in all the previous theorems indeed imply a bi-Hölder behaviour of the map $\mu\mapsto T_\mu$ on subsets of $\Prob_2(\Rsp^d)$ with respect to both Wasserstein distances $\Wass_1$ and $\Wass_2$.
\end{remark}

\begin{remark}[Constants]
The constants appearing in the above theorem may all be tracked down and all feature the product of three terms that depend respectively on the dimension $d$, the diameter and perimeter of $\X$, and the bounds $m_\rho, M_\rho > 0$ on $\rho$ that are such that $m_\rho \leq \rho \leq M_\rho$ on $\X$. If $\mu,\nu$ are supported on a compact set $\Y$, the constants also feature a factor that only depends on the smallest positive real $R_\Y$ such that $\Y \subset B(0, R_\Y)$. For instance in such compact setting, the constant controlling the $\L^2(\rho)$ distance between $\phi_\mu$ and $\phi_\nu$ reads:
$$ C_{d,p,\X,\rho,M_p} = C_{d,p,\X,\rho,\Y} = e(d+1)2^d \frac{M_\rho^2}{m_\rho^2} \diam(\X)^2 R_\Y. $$
In the non-compact setting, a factor involving $M_p$ appears, as well as a factor involving the Poincaré constant of order $p$ of $\X$ and the $p$-th power of the ratio $\frac{R_\X}{r_\X}$, where $r_\X, R_\X > 0$ are the largest and smallest reals such that $B(0, r_\X) \subset \X \subset B(0, R_\X)$ (assuming without any loss of generality $\X$ contains the origin).
\end{remark}

A large class of probability measures verifies the moment assumption, such as sub-Gaussian or sub-exponential measures (see Remark \ref{rk:sub-exponential}). 
A preliminary version of this theorem was announced in \cite{pmlr-v108-merigot20a}, with a different proof strategy, relying on the study of the case where both $\mu,\nu$ are supported on the same finite set. The proof in \cite{pmlr-v108-merigot20a} led to a worse Hölder exponent in the compact case, and couldn't deal with non-compactly supported measures. We do not know whether the Hölder exponents in this theorem are optimal.

To prove these stability estimates, we use the fact that the dual potentials solve a convex minimization problem involving the functional $\Kant(\psi) = \int\psi^*\dd\rho$, which we call \emph{Kantorovich's} functional. We first prove in (\S\ref{sec:strong-convexity-K}) a strong convexity estimate for Kantorovich's functional, relying in particular on the Brascamp-Lieb inequality, and which holds under the assumption that the Brenier potentials are bounded. This strong convexity estimate is then translated into a stability estimate concerning the dual and Brenier potentials (\S\ref{sec:stab-dual}). The stability of Brenier maps is then obtained (\S\ref{sec:stab-ot-maps}), relying in particular on a Gagliardo-Nirenberg 
type inequality for the difference of convex functions (\S\ref{sec:ineg-convex-functions}), which might be of independent interest.

\section{Strong convexity of Kantorovich's functional}
\label{sec:strong-convexity-K}

Let $\X$ be a compact convex subset of $\Rsp^d$, and let $\rho$ be a probability density on $\X$.
Given any measure $\mu \in\Prob_2(\Rsp^d)$, we consider the problem of finding the coupling $\gamma \in\Pi(\rho,\mu)$ which maximizes the  correlation $\int \sca{x}{y}\dd\gamma(x,y)$.
This problem is equivalent to the standard quadratic optimal transport problem and in this setting 
Kantorovich duality reads
$$ \max_{\gamma \in\Pi(\rho,\mu)} \int \sca{x}{y}\dd\gamma(x,y) = \min_{\psi\in\Class^0(\Rsp^d)} \Kant(\psi) + \int_{\Rsp^d} \psi \dd\mu, $$
where the functional $\Kant$, which we will call \emph{Kantorovich's functional}, is defined by 
\begin{align*}
    \Kant(\psi) := \int_{\X} \psi^* \dd \rho.
\end{align*}
This dual formulation of the maximal correlation problem can for instance be found as Particular Case 5.16 in  \cite{villani2008optimal}. 
Kantorovichs' functional is convex because for any $x\in \Rsp^d$, the map $\psi\mapsto \psi^*(x)$ is convex in $\psi$. Moreover, formal computations, which
are justified in Proposition \ref{prop:Poincare-smooth}, show that $$\nabla \Kant(\psi) = - (\nabla \psi^*)_\# \rho.$$ 
In particular, with $\phi_\mu$ the Brenier potential associated to the optimal transport problem between  $\rho$ and $\mu$ and  $\psi_\mu = \phi_\mu^*$ its convex conjugate, this gives the relation $\psi_\mu = (\nabla \Kant)^{-1}(-\mu)$. Since $\Kant$ is convex, its gradient must be monotone, thus implying  that for all probability measures 
$\mu^0,\mu^1\in \Prob_2(\Rsp^d)$, 
$$ \sca{\psi_{\mu^1} - \psi_{\mu^0}}{\nabla \Kant(\psi_{\mu^1}) - \nabla \Kant(\psi_{\mu^0})} = \sca{\psi_{\mu^1} - \psi_{\mu^0}}{\mu^0 - \mu^1} \geq 0.$$
Our aim in this section is to prove \cref{th:Poincare-ineq}, establishing strong convexity estimates for Kantorovich's functional $\Kant$, which we will later be able to translate into stability estimates for $\mu\mapsto\psi_\mu = (\nabla \Kant)^{-1}(-\mu).$

\begin{theorem}[Strong convexity] 
\label{th:Poincare-ineq}
Let $\mu^0, \mu^1 \in \Prob_2(\Rsp^d)$ and let $\rho$ be a probability density over a  compact convex set $\X$, satisfying $0 < \mrho \leq \rho \leq \Mrho$. For $k \in \{0, 1\}$, denote $\phi^k = \phi_{\mu^k}$ the Brenier potential between $\rho$ and $\mu^k$ (see Definition~\ref{def:Brenier}). Assume that 
\begin{equation}\label{eq:hyp-phi}
    \forall k\in\{0,1\},\quad  -\infty < \mphi \leq  \min_\X \phi^k \leq \max_\X \phi^k \leq \Mphi < +\infty.
\end{equation}
Then the convex conjugates $\psi^0$ and $\psi^1$ of $\phi^0$ and $\phi^1$ verify:
\begin{equation}
\label{eq:Poincaré-psi}
\Var_{\frac{1}{2}(\mu^0+\mu^1)} (\psi^1 - \psi^0) \leq C_d \frac{ \Mrho^2}{\mrho^2} (\Mphi - \mphi) \sca{\psi^0 - \psi^1}{\mu^1 - \mu^0}, 
\end{equation}
where $C_d = e(d+1)2^{d-1}$.
\end{theorem}

\begin{remark}[Variance] The left-hand side of \eqref{eq:Poincaré-psi} involves the variance of $\psi^1-\psi^0$ instead of a squared $\L^2$ norm. This is to be expected, because of the invariance of the Kantorovich's functional under addition of a constant. The choice of $\mu_0+\mu_1$ as the reference measure for the variance term in inequality \eqref{eq:Poincaré-psi} may seem unnatural, but we note that there is no natural reference measure on the target. The choice of $\frac{1}{2}(\mu^0+\mu^1)$ as the reference measures proves relevant for establishing the stability of Brenier potentials in the next section. Proposition~\ref{prop:bound-lp-primal-dual} especially asserts that $\Var_{\frac{1}{2}(\mu^0+\mu^1)}(\psi^1 - \psi^0) \geq \frac{1}{2}\Var_\rho(\phi^1 - \phi^0)$. We also note that, as detailed in the proof of Theorem \ref{th:Poincare-ineq}, 
the left-hand side of the inequality could actually be replaced by the quantity
$$ C_d \frac{\Mrho}{\mrho} \int_0^1 \Var_{\mu^t}(\psi^1 - \psi^0)\dd t,$$
where for $t \in [0, 1]$, $\mu^t = \nabla((1-t)\psi^0 + t \psi^1)^*_{\#} \rho$ interpolates between $\mu^0$ and $\mu^1$. This inequality is tighter, but the interpolation $t\mapsto \mu_t$ has no simple interpretation and is quite difficult to manipulate. In particular, this curve is \emph{not} a generalized geodesic in the sense of Ambrosio, Gigli, Savaré \cite{ambrosio2008gradient}.
\end{remark}

\begin{remark}[Optimality of exponents]
Estimate \eqref{eq:Poincaré-psi} is optimal in term of exponent of $\Var_{\frac{1}{2}(\mu^0+\mu^1)} (\psi^1 - \psi^0)$. Indeed in dimension $d=1$, for $\epsilon \geq 0$, denote $\mu_\epsilon$ the uniform probability measure on the segment $[\epsilon, 1+\epsilon]$. Then for $\rho = \mu_0$, one can show that for $\epsilon \leq 1$, both $\Var_{\frac{1}{2}(\mu^0+\mu^\epsilon )} (\psi^\epsilon - \psi^0)$ and $\sca{\psi^0 - \psi^\epsilon}{\mu^\epsilon - \mu^0}$ are of the order of $\epsilon^2$.
\end{remark}

The strong convexity estimate of Theorem \ref{th:Poincare-ineq} may find applications beyond the stability of optimal transport maps. In particular, the authors noticed in a subsequent work \cite{barycenters} that this estimate can be used to derive non-trivial quantitative stability bounds for Wasserstein barycenters (defined in \cite{agueh:hal-00637399}) under mild regularity assumptions on the marginal measures.

We also note that a strong convexity estimate similar to \eqref{eq:Poincaré-psi} was derived in \cite{pmlr-v151-delalande22a} for entropy-regularized optimal transport, using a proof similar in spirit to the proof of Theorem \ref{th:Poincare-ineq} to be presented below.%

The strong convexity  estimate \eqref{eq:Poincaré-psi} is derived from a local estimate, 
a Poincaré-Wirtinger inequality for the second derivative of $\Kant$, which is in turn a consequence of the Brascamp-Lieb inequality \eqref{eq:brascamp-lieb}. 
To make the connection with the Brascamp-Lieb inequality clearer, we first compute the first and second order derivatives of $\Kant$ along the path 
$((1-t)\psi^0 + t \psi^1)_{t \in [0, 1]}$, under regularity and strong convexity hypotheses. These hypothesis are  relaxed in the proof of Theorem~\ref{th:Poincare-ineq}.

\begin{proposition} \label{prop:Poincare-smooth} 
Let $\phi^0, \phi^1 \in \Class^2(\Rsp^d)$ be strongly convex functions. Define $\psi^0 = (\phi^0)^*$, $\psi^1 = (\phi^1)^*$ and $v = \psi^1-\psi^0$. For $t \in [0, 1]$, define $\psi^t = \psi^0 + t v$ and finally $\phi^t = (\psi^t)^*$. Then, $\phi^t$ is a strongly convex function, belongs to $\Class^2(\Rsp^d)$, and
\begin{align}
\frac{\dd}{\dd t} \Kant(\psi^t) &=  -\int_\X v(\nabla \phi^t(x)) \dd \rho(x), \\
\frac{\dd^2}{\dd t^2} \Kant(\psi^t) &= \int_\X \sca{\nabla v(\nabla\phi^t(x))}{\DD^2 \phi^t(x)\cdot \nabla v(\nabla \phi^{t}(x)}\dd \rho(x).
    \label{eq:hessian_kant}
\end{align}
\end{proposition}

We then find a positive lower-bound on the second order derivative expressed in equation \eqref{eq:hessian_kant} using the Brascamp-Lieb inequality \cite{brascamp-lieb}. We cite here a version of this inequality that is adapted to our context, i.e. that concerns log-concave probability measures supported on the compact and convex set $\X$. This statement is a special case of Corollary 1.3 of \cite{lepeutrec:hal-01349786}, where $\X$ is a convex subset of a Riemannian manifold. We also refer to Section 3.1.1 of \cite{Kolesnikov2017}.

\begin{theorem}[Brascamp-Lieb inequality]
\label{th:brascamp-lieb}
Let $\phi \in \Class^2(\X)$ be a strictly convex function. Let $\tilde{\rho}$ be the probability measure defined by $\dd \tilde{\rho} = \frac{1}{Z_\phi}
\exp(-\phi) \dd x$ with $Z_\phi = \int_\X \exp(-\phi) \dd x$. Then every smooth function $s$ on $\X$ verifies:
\begin{equation}
     \label{eq:brascamp-lieb}
     \Var_{\tilde{\rho}} (s) \leq \Esp_{\tilde{\rho}} \sca{\nabla s }{ (\DD^2 \phi)^{-1} \cdot \nabla s }.
 \end{equation}
\end{theorem}

We now justify the computation of the derivatives presented in Proposition \ref{prop:Poincare-smooth}.

\begin{proof}[Proof of Proposition~\ref{prop:Poincare-smooth}] 
We assume that $\phi^0,\phi^1$ are both $\alpha$-strongly convex and belong to $\Class^2(\Rsp^d)$. Then, the convex conjugates $\psi^0 = (\phi^0)^*$, $\psi^1 = (\phi^1)^*$ are $\Class^2$ with $1/\alpha$-Lipschitz gradients and satisfy $\DD^2 \psi^0 > 0, \DD^2 \psi^1 > 0$ everywhere on $\Rsp^d$. Hence their linear interpolates $\psi^{t} = (1-t)\psi^0 + t\psi^1$ enjoy the same properties.
This in turn implies that for all $t\in[0,1],$  the  convex conjugate $\phi^{t}$ of $\psi^t$ belongs to $\Class^2(\Rsp^d)$ and is $\alpha$-strongly convex.

We will now prove that the map $G:(t,x) \mapsto \nabla \phi^t(x)$ has class $\Class^1$. Let $F: [0,1]\times\Rsp^d\times \Rsp^d\to\Rsp^d$ 
be the continuously differentiable function defined by $F(t,x,y) = \nabla \psi^t(y) - x$. A well-known property of the convex conjugate is that $\nabla \phi^t$ is the inverse of $\nabla \psi^t$, implying that $G(t,x)$ is uniquely characterized
by $F(t,x,G(t,x)) = 0$. Since $\DD^2 \psi^t > 0$, the Jacobian $\DD_y F(t, x, y) = \DD^2\psi^{t}(y)$ is invertible and the implicit function theorem thus implies that $G$ has class $\Class^1$. Differentiating  the relation $F(t,x,G(t,x)) = 0$ with respect to time, we get
\begin{equation}
    \label{eq:gradient-nabla-phi-t}
    \frac{\dd}{\dd t} \nabla \phi^t(x) = - \DD^2\phi^t(x) \cdot \nabla v(\nabla \phi^t(x)).
\end{equation}
By Fenchel-Young's equality case, one has  for any $x\in \X$ and $t\in[0,1]$,
$$
\phi^t(x) = \sca{x}{\nabla \phi^t(x)} - \psi^{t}(\nabla \phi^t(x)),$$
so that $\phi^t$ is at least $\Class^1$ with respect to time. We can actually differentiate this equation with respect to time twice and 
using \eqref{eq:gradient-nabla-phi-t} we get
$$ 
\begin{aligned}
&\frac{\dd}{\dd t}\phi^t(x) = \sca{x}{\frac{\dd}{\dd t}\nabla \phi^t(x)} - v(\nabla \phi^t(x)) - \sca{\nabla \psi^t(\nabla \phi^t)}{\frac{\dd}{\dd t}\nabla \phi^t(x)} = - v(\nabla \phi^t(x)), \\
&\frac{\dd^2}{\dd t^2}\phi^t(x) = - \sca{\nabla v(\nabla\phi^t(x))}{\frac{\dd}{\dd t} \nabla \phi^t(x)} = \sca{\nabla v(\nabla\phi^t(x))}{\DD^2 \phi^t(x)\cdot \nabla v(\nabla \phi^{t}(x))}.
\end{aligned}
$$
Since $\Kant(\psi^t) = \int_\X \phi^t(x)\dd \rho(x)$, we get the result by differentiating twice under the integral.

\end{proof}

\begin{proposition} \label{prop:Poincare-ineq-smooth}
In addition to the assumptions of \cref{th:Poincare-ineq}, assume that the Brenier potentials $\phi^0,\phi^1$ are strongly convex, belong to $\Class^2(\Rsp^d)$, and that 
$\nabla\phi^0$ and $\nabla \phi^1$ induce diffeomorphisms between $\X$ and a closed ball $\Y$. Then, inequality \eqref{eq:Poincaré-psi} holds.
\end{proposition}

\begin{proof} Under the assumptions on $\phi^0,\phi^1$,
Proposition \ref{prop:Poincare-smooth} ensures that the function $\phi^t$ it defines is strongly convex and belongs to $\Class^2(\Rsp^d)$  for any $t \in [0, 1]$. By the fundamental theorem of calculus, again with the notations of Proposition \ref{prop:Poincare-smooth}, we have:
\begin{align}
    \sca{\psi^0 - \psi^1}{\mu^1 - \mu^0} = \left.\frac{\dd}{\dd t} \Kant(\psi^t)\right|_{t=1} - \left.\frac{\dd}{\dd t} \Kant(\psi^t)\right|_{t=0} 
    = \int_0^1 \frac{\dd^2}{\dd t^2} \Kant(\psi^t) \dd t. 
    \label{eq:int-der-second-K}
\end{align}
From Proposition \ref{prop:Poincare-smooth}, we have the following expression for the second derivative of $\Kant$:
$$ \frac{\dd^2}{\dd t^2} \Kant(\psi^t) = \Esp_{\rho} \sca{\nabla v (\nabla \phi^t)}{ (\DD^2 \phi^t) \cdot \nabla v(\nabla \phi^t) }.$$
We introduce $\tilde{v}^t = v (\nabla \phi^t)$ for any $t \in [0, 1]$, which belongs to $\Class^1(\Rsp^d)$ as the composition of $v = \psi^1 - \psi^0 \in \Class^{2}(\Rsp^d)$ and $\nabla \phi^t$. We have $\nabla \tilde{v}^t = \DD^2 \phi^t \cdot \nabla v(\nabla \phi^t)$, where $(\DD^2 \phi^t)$ is invertible by strong convexity. Thus,
\begin{equation}
    \label{eq:2nd-derivative-Kant}
    \frac{\dd^2}{\dd t^2} \Kant(\psi^t) = \Esp_{\rho} \sca{\nabla \tilde{v}^t }{ (\DD^2 \phi^t)^{-1} \cdot \nabla \tilde{v}^t }.
\end{equation}
We now introduce $\tilde{\rho}^t = \exp(-\phi^t)/Z_t$ where $Z_t = \int_\X \exp(-\phi^t(x)) \dd x$, which is the density of a log-concave probability measure supported on $\X$. The Brascamp-Lieb inequality, recalled in \cref{th:brascamp-lieb}, then ensures that
 \begin{equation}
     \label{eq:brascamp-lieb-2}
     \Var_{\tilde{\rho}^t} (\tilde{v}^t) \leq \Esp_{\tilde{\rho}^t} \sca{\nabla \tilde{v}^t }{ (\DD^2 \phi^t)^{-1} \cdot \nabla \tilde{v}^t }.
 \end{equation}
 We assumed that for any $k \in \{0, 1\}$ and $x \in \X$, $\mphi \leq \phi^k(x) \leq \Mphi$. We claim that this property is transferred to $\phi^t$ for any $t \in [0, 1]$. Indeed,
 on the one hand for all $t \in [0,1]$,
\begin{align*}
    \phi^t &= \left( (1-t) \psi^0 + t \psi^1 \right)^* 
    \leq (1-t) (\psi^0)^* + t (\psi^1)^* 
    = (1-t) \phi^0 + t \phi^1 
    \leq \Mphi,
\end{align*}
where we used the convexity of the convex conjugation. On the other hand, for any $x \in \X$, we have by definition:
\begin{align*}
    \phi^t(x) &= \sup_{y \in \Rsp^d} \sca{x}{y} - \psi^t(y)
    \geq - \psi^t(0) 
    = -(1-t) \psi^0(0) - t\psi^1(0).
\end{align*}
But again, for $k \in \{0, 1\},$ $\psi^k(0) = \sup_{x \in \X} - \phi^k(x) \leq - \mphi$, ensuring that  $\phi^t \geq \mphi$ for all $t \in [0, 1]$.
The inequality $\mphi \leq \phi^t \leq \Mphi$ allows us to compare the densities $\rho$ and $\tilde{\rho}^t$:
$$ \left(\frac{\exp(-\Mphi)}{\Mrho Z_t}\right) \rho \leq \tilde{\rho}^t \leq \left(\frac{\exp(-\mphi)}{\mrho Z_t}\right) \rho. $$ This comparison and equation \eqref{eq:brascamp-lieb-2} thus give:
\begin{equation*}
    \left(\frac{\exp(-\Mphi)}{\Mrho Z_t}\right) \Var_{\rho} (\tilde{v}^t) \leq  \left(\frac{\exp(-\mphi)}{\mrho Z_t}\right) \Esp_{\rho} \sca{\nabla \tilde{v}^t }{ (\DD^2 \phi^t)^{-1} \cdot \nabla \tilde{v}^t },
\end{equation*}
where we used that for any absolutely continuous $\rho_1, \rho_2 \in \Prob_2(\Rsp^d)$ and $f : \Rsp^d \to \Rsp$, the density comparison $\rho_1 \leq C \rho_2$ for some $C > 0$ yields $$ \Var_{\rho_1}(f) = \min_{c \in \Rsp} \nr{f - c}^2_{\L^2(\rho_1)} \leq C \min_{c \in \Rsp} \nr{f - c}^2_{\L^2(\rho_2)} = C \Var_{\rho_2}(f). $$ Therefore, using $\tilde{v}^t = v(\nabla \phi^t)$, $\mu^t = (\nabla \phi^t)_\# \rho$, $v = \psi^1 - \psi^0$ and expression \eqref{eq:2nd-derivative-Kant}:
\begin{equation}
    \label{eq:brascamp-lieb-3}
    \Var_{\mu^t}(\psi^1 - \psi^0) \leq \frac{\Mrho}{\mrho} \exp(\Mphi - \mphi) \frac{\dd^2}{\dd t^2} \Kant(\psi^t).
\end{equation}
Note that $\mu^t = \nabla((1-t)\psi^0 + t \psi^1)^*_{\#} \rho$ interpolates between $\mu^0$ in $t=0$ and $\mu^1$ in $t=1$, but this interpolation is neither a displacement interpolation in the sense of McCann \cite{MCCANN1997153} nor a generalized geodesic in the sense of Ambrosio, Gigli, Savaré \cite{ambrosio2008gradient}. Recalling equation \eqref{eq:int-der-second-K}, this equation is similar to that of \eqref{eq:Poincaré-psi}, except that we would like to replace $\mu^t$ by $\frac{1}{2}(\mu^0+\mu^1)$. For this purpose, we  will prove that
\begin{equation}\label{eq:lb-mut}
\mu^t \geq \frac{\mrho}{\Mrho}  \min(t,1-t)^d (\mu^0+\mu^1).
\end{equation}
This will be done using an explicit expression for $\mu^t$. By smoothness and strong convexity of the function $\phi^t$, the restriction of $\nabla \phi^t$ to $\X$ is a diffeomorphism on its image. This implies that $\mu^t$ is absolutely continuous with respect to the Lebesgue measure. Moreover, by e.g. Villani~\cite[p.9]{villani2003}, for any $x\in \X$ the density of $\mu^t$ with respect to Lebesgue,  also noted $\mu^t$,  is given by
$\mu^t(\nabla \phi^t(x)) \det (\DD^2 \phi^t(x)) = \rho(x)$.
Setting $y = \nabla \phi^t(x)$ in this formula, we get
$$ \forall y \in \nabla \phi^t(\X),\quad \mu^t(y) = \rho(\nabla \psi^t(y)) \det(\DD^2 \psi^t(y)). $$

By assumption, $\nabla\phi^k$ is a diffeomorphism from $\X$ to $\Y$ and so is $\nabla \psi^k$ from $\Y$ to $\X$. Thus by convexity of $\X$, $\nabla \psi^t(\Y) \subset \X$, which entails $\Y \subset \nabla \phi^t(\X)$. The equality above then gives
$$ \forall k\in \{0,1\},\forall y\in \Y,\quad \mu^k(y) \leq \Mrho \det(\DD^2 \psi^k(y)). $$
On the other hand, the same equality gives
$$ \forall t\in[0,1], \forall y\in \Y, \quad \mu^t(y) \geq \mrho \det(\DD^2 \psi^t(y)). $$
Using the two inequalities above and the  concavity of $\det^{1/d}$ over the set of non-negative symmetric matrices, we get for every 
$y \in \Y,$
\begin{align*}
    \mu^t(y) &\geq \mrho \det(\DD^2\psi^t(y))\\
    &\geq \mrho \left((1-t) \det(\DD^2\psi^0)^{1/d} + t \det(\DD^2\psi^1)^{1/d}\right)^d\\
    &\geq \mrho \min(t,1-t)^d (\det(\DD^2\psi^0(y)) + \det(\DD^2\psi^1(y)))\\
    &\geq \frac{\mrho}{\Mrho} \min(t,1-t)^d (\mu^0(y) + \mu^1(y))
\end{align*}
Using that $\spt(\mu^0) = \spt(\mu^1) = \Y$, this directly implies \eqref{eq:lb-mut},

which in turn gives us
\begin{equation*}
    \Var_{\mu^t}(v) \geq  2\min(t,1-t)^d \frac{\mrho}{\Mrho} \Var_{\frac{1}{2}(\mu^0+\mu^1)}(v).
\end{equation*}
Combined with inequality \eqref{eq:brascamp-lieb-3}, this gives after integrating over $t \in [0, 1]$:
\begin{equation*}
    \frac{1}{(d+1)2^{d-1}} \frac{\mrho}{\Mrho} \Var_{\frac{1}{2}(\mu^0+\mu^1)} (v) \leq \frac{ \Mrho}{\mrho} \exp(\Mphi - \mphi) \int_0^1 \frac{\dd^2}{\dd t^2} \Kant(\psi^t) \dd t.
\end{equation*}
Using \eqref{eq:int-der-second-K}, we obtain the inequality
\begin{equation}
\label{eq:estimate-exp}
    \Var_{\frac{1}{2}(\mu^0+\mu^1)} (\psi^1 - \psi^0) \leq (d+1)2^{d-1} \frac{ \Mrho^2}{\mrho^2} \exp(\Mphi - \mphi) \sca{\psi^0 - \psi^1}{\mu^1 - \mu^0}.
\end{equation}
We finally leverage an in-homogeneity in the scale of the Brenier potentials in the last inequality in order to improve the dependence on $\Mphi - \mphi$. For any $\lambda>0$, introduce for $k \in \{0, 1\}$ the Brenier potential $\phi^k_\lambda = \lambda \phi^k$ and denote $\mu^k_\lambda = (\nabla \phi^k_\lambda)_\# \rho$ the corresponding probability measure and $\psi^k_\lambda = (\phi^k_\lambda)^*$ its dual potential. Then using the formula $\psi^k_\lambda = \lambda \psi^k(\cdot/\lambda)$, one can notice that for any $\lambda>0$,
\begin{gather*}
    \Var_{\frac{1}{2}(\mu^0_\lambda+\mu^1_\lambda )} (\psi^1_\lambda - \psi^0_\lambda) = \lambda^2 \Var_{\frac{1}{2}(\mu^0+\mu^1)} (\psi^1 - \psi^0), \\
    \sca{\psi^0_\lambda - \psi^1_\lambda}{\mu^1_\lambda - \mu^0_\lambda} = \lambda \sca{\psi^0 - \psi^1}{\mu^1 - \mu^0}, \\
    \forall x\in\X, \forall k\in\{0, 1\}, \quad \lambda \mphi \leq \phi^k_\lambda(x) \leq \lambda \Mphi.
\end{gather*}
Thus applying inequality \eqref{eq:estimate-exp} to $\mu^0_\lambda, \mu^1_\lambda$ and the associated potentials yields for any $\lambda>0$
\begin{equation*}
    \Var_{\frac{1}{2}(\mu^0+\mu^1)} (\psi^1 - \psi^0) \leq (d+1)2^{d-1} \frac{ \Mrho^2}{\mrho^2} \frac{\exp(\lambda(\Mphi - \mphi))}{\lambda} \sca{\psi^0 - \psi^1}{\mu^1 - \mu^0}.
\end{equation*}
Choosing $\lambda = \frac{1}{\Mphi - \mphi}$ in the last inequality finally gives
\begin{equation*}
    \Var_{\frac{1}{2}(\mu^0+\mu^1)} (\psi^1 - \psi^0) \leq e (d+1)2^{d-1} \frac{ \Mrho^2}{\mrho^2} (\Mphi - \mphi) \sca{\psi^0 - \psi^1}{\mu^1 - \mu^0}. \qedhere
\end{equation*}

\end{proof}

To deduce the general case of \cref{th:Poincare-ineq}, we need to approximate the convex potentials $\phi^0, \phi^1$ on $\X$ with strongly convex potentials $\phi^0_n, \phi^1_n$ that belong to $\Class^2(\X)$ and that are such that their gradients $\nabla \phi^0_n, \nabla \phi^1_n$ induce diffeomorphisms between $\X$ and a closed ball $\Y_n$. A regularization that uses a (standard) convolution does not seem directly feasible. Indeed, $\phi^k$ is defined on $\X$ only, and 
its gradient explodes on the boundary of $\X$ when $\mu^k$ has non-compact support, so that any convex extension of $\phi^k$ to $\Rsp^d$ has to take value $+\infty$. 

Our strategy is as follows. First, we resort to Moreau-Yosida's regularization to approximate the functions $\phi^0,\phi^1$ by regular convex functions
defined on $\Rsp^d$. Then, we regularize the target probability measures associated to the approximated potentials and resort to Caffarelli’s regularity theory  to guarantee  smoothness and strong convexity. Caffarelli’s regularity theory results  require smoothness assumptions on the source probability measure and strong convexity and smoothness assumption on the domain. We make these assumptions in the next proposition, but we will later show that these can be relaxed to get the general case of \cref{th:Poincare-ineq}.

\begin{proposition}
\label{prop:reg-psi}
Let $\mu^0, \mu^1 \in \Prob_2(\Rsp^d)$. Let $\X$ be a compact, smooth and strongly convex set, let $\rho$ be a smooth probability density on $\X$ and assume that $\rho$ is bounded away from zero and
infinity on this set.
Denote $\phi^k$ the  Brenier potentials for the quadratic optimal transport from $\rho$ to $\mu^k$, $\psi^k = (\phi^k)^*$, and assume that there exists 
$\mphi, \Mphi \in \Rsp$ such for $k\in \{0, 1\}$ and any $x \in \X$, $$ \mphi \leq \phi^k(x) \leq \Mphi.$$
Then there exists sequences of strongly convex functions $(\phi^0_n)_{n\in \Nsp}, (\phi^1_n)_{n\in\Nsp}$ in $\Class^2(\Rsp^d)$  such that if one
introduces $\mu^k_n = \nabla \phi^k_{n\#} \rho$ and $\psi^k_n = (\phi^k_n)^*$, then:
\begin{enumerate}[label=(\roman*)]
    \item \label{it:cv-potentials} $\lim_{n\to +\infty} \sca{\psi^0_n - \psi^1_n}{\mu^1_n - \mu^0_n} = \sca{\psi^0 - \psi^1}{\mu^1 - \mu^0}$,
    \item \label{it:cv-var} $\lim_{n\to +\infty}  \Var_{\frac{1}{2}(\mu^0_n+\mu^1_n)}(\psi^1_n - \psi^0_n) = \Var_{\frac{1}{2}(\mu^0+\mu^1)}(\psi^1 - \psi^0)$,
    \item \label{it:bound-phi} let ${\mphi}_n= \min_{\X} \min_k \phi^k_n$, and ${\Mphi}_n = \max_{\X} \max_k \phi^k_n$. Then,
    $$\mphi \leq \liminf_{n\to +\infty} {\mphi}_n \leq \limsup_{n\to +\infty} {\Mphi}_n \leq \Mphi$$
      \item \label{it:support} there exists a closed ball  $\Y_n$ such that 
        for  $k\in\{0,1\}$, $\nabla \phi^k_n$ is a
        diffeomorphism between $\X$ and $\Y_n$.
\end{enumerate}
\end{proposition}

Before proving this proposition, we recall some facts regarding Moreau-Yosida's regularization of convex functions. Quoting Section 3.4 of
\cite{attouch1984variational},  the
Moreau-Yosida regularization of parameter $\lambda > 0$ of a closed
and proper convex function $f: \Rsp^d \to \Rsp \cup\{+\infty\}$ is
defined for all $x \in \Rsp^d$ by infimum convolution of the function
$f$ with $\frac{1}{2\lambda} \nr{\cdot}^2$:
\begin{align*}
    f_\lambda(x) &= \min_{u \in \Rsp^d} f(u) + \frac{1}{2\lambda} \nr{u - x}^2.
\end{align*}
The next lemma gathers a few properties of the Moreau-Yosida regularisation.

\begin{lemma}\label{lemma:MY}
Let $f : \Rsp^d \to \Rsp \cup \{+\infty\}$ be a closed and proper convex function and let $\lambda > 0$. Then,
  \begin{enumerate}[label=(\roman*)]
    \item \label{it:MY-conjugate} $f_\lambda = (f^* + \frac{\lambda}{2} \nr{\cdot}^2)^*$,
    \item \label{it:MY-pointwise-lim} for all $x \in \Rsp^d,$ $\lim_{\lambda \to 0} f_\lambda (x) = f(x), $
    \item \label{it:MY-lip-gradient} $f_\lambda \in \Class^{1,1}(\Rsp^d)$ and more
      precisely, $\nabla f_\lambda$ is  $\frac{1}{\lambda}$-Lipschitz,
    \item \label{it:cv-grad-MY}if $f$ is differentiable at  $x \in \Rsp^d$, then $\lim_{\lambda \to 0} \nabla f_\lambda (x) = \nabla f(x),$
    \item \label{it:bound-grad-MY}if $f$ is differentiable at  $x \in \Rsp^d$, then  $\nr{\nabla f_\lambda(x)} \leq \nr{\nabla f(x)}$.
  \end{enumerate}
\end{lemma}

\begin{proof}
Point \ref{it:MY-conjugate} is found in Proposition 3.3 of \cite{attouch1984variational}, points \ref{it:MY-pointwise-lim} and \ref{it:MY-lip-gradient} are found in Theorem 3.24 of \cite{attouch1984variational} and \ref{it:cv-grad-MY} and \ref{it:bound-grad-MY} can be found in Proposition 2.6 of \cite{brezis1973ope}.
\end{proof}

\begin{proof}[Proof of Proposition \ref{prop:reg-psi}]
  \emph{First regularization and truncation.}  We will first
  approximate and extend the Brenier potentials $\phi^0, \phi^1$,
  which are defined on $\X$, with elements of
  $\Class^{1,1}(\Rsp^d)$. To do so, we extend $\phi^0,\phi^1$ by
  $+\infty$ outside of the set $\X$ and for any $\alpha>0$ we denote
  by $\phi^k_\alpha$ the Moreau-Yosida regularization of $\phi^k$ with
  parameter $\alpha$. We let $\mu^k_\alpha = (\nabla
  \phi^k_{\alpha})_{\#}\rho$ and define $\psi^k_\alpha$ as the convex
  convex conjugate of $\phi^k_\alpha$. By Lemma~\ref{lemma:MY},
  $\nabla\phi^k_\alpha$ is Lipschitz on the bounded domain $\X$,
  implying that the images of $\nabla \phi^k_\alpha(\X)$ are contained
  in a closed ball $\Y_\alpha = B(0,R_\alpha)$. We now prove the
  claimed convergences \ref{it:cv-potentials}---\ref{it:bound-phi}, relying mainly on the dominated convergence theorem.
  We first note that if $f: \Rsp^d\to\Rsp$ satisfies the growth 
  condition $\abs{f(x)} \leq C(1+\nr{x}^2)$ for some constant $C$,
  \begin{equation} \label{eq:dominated-convergence}
       \sca{f}{\mu^k_\alpha}  = \int_{\X} f(\nabla \phi^k_\alpha) \dd \rho \xrightarrow[\alpha\to 0]{}  
  \int_{\X} f(\nabla \phi^k) \dd \rho = \sca{f}{\mu^k}.
  \end{equation}
Indeed, \cref{lemma:MY}.\ref{it:cv-grad-MY} ensures that for every point $x\in\X$ where $\phi^k$ is differentiable, thus for $\rho$-almost every point $x$, 
one has $\lim_{\alpha\to 0} \nabla \phi^k_\alpha(x) = \nabla \phi^k(x)$. Besides,  for all such $x$, \cref{lemma:MY}.\ref{it:bound-grad-MY} gives
$$ \abs{f(\nabla \phi^k_\alpha(x))} \leq C\left(1+\nr{\nabla \phi^k_\alpha(x)}^2\right) \leq C\left(1+\nr{\nabla \phi^k(x)}^2\right).$$
Moreover, 
$$ \int_{\X} \left( 1+\nr{\nabla \phi^k(x)}^2 \right) \dd\rho(x) \leq 1+M_2(\nabla \phi^k(x)_{\#} \rho) = 1 + M_2(\mu^k) < +\infty.$$
Thus, the dominated convergence theorem ensures that \eqref{eq:dominated-convergence} holds.

\ref{it:cv-potentials} By \cref{lemma:MY}.\ref{it:MY-conjugate}, $\psi^0_\alpha - \psi^1_\alpha = \phi^{0*}_\alpha - \phi^{1*}_\alpha = \psi^0 + \frac{\alpha}{2}\nr{\cdot}^2 - \psi^1 - \frac{\alpha}{2}\nr{\cdot}^2 = \psi^0 - \psi^1$. Since $\psi^0,\psi^1$ are convex conjugates of functions defined
on the compact set $\X$, the functions $\psi^0$ and $\psi^1$ are
(globally) Lipschitz on $\Rsp^d$. Thus $f = \psi^0_\alpha - \psi^1_\alpha$  is also Lipschitz, and therefore satisfies a growth condition of 
the form $\abs{f}\leq C(1+\nr{x})$. By an application of \eqref{eq:dominated-convergence}, we get 
$$
\lim_{\alpha\to 0} \sca{\psi^0_\alpha -  \psi^1_\alpha}{\mu^1_\alpha - \mu^0_\alpha} =
 \lim_{\alpha\to 0} \sca{\psi^0 -  \psi^1}{\mu^1_\alpha - \mu^0_\alpha} = \sca{\psi^0 - \psi^1}{\mu^1 - \mu^0}. $$

\ref{it:cv-var}  We use $\Var_\mu(f) = \int f^2 \dd\mu - (\int f\dd\mu)^2$. Letting $f$ as in the previous item, we get
\begin{align*}
    \begin{aligned}
    &
    \Var_{\frac{1}{2}(\mu^0_\alpha+\mu^1_\alpha)} (\psi^1_\alpha - \psi^0_\alpha) = \sca{f^2}{\frac{1}{2}(\mu^0_\alpha+\mu^1_\alpha)}  - \sca{f}{\frac{1}{2}(\mu^0_\alpha+\mu^1_\alpha)}^2,\\
    &\Var_{\frac{1}{2}(\mu^0+\mu^1)} (\psi^1 - \psi^0) = \sca{f^2}{\frac{1}{2}(\mu^0+\mu^1)}  - \sca{f}{\frac{1}{2}(\mu^0+\mu^1)}^2.
    \end{aligned}
\end{align*}
Since $f$ is Lipschitz, both $f$ and $f^2$ satisfy the growth condition allowing us to apply \eqref{eq:dominated-convergence}. We therefore get
$$ \lim_{\alpha\to 0} \Var_{\frac{1}{2}(\mu^0_\alpha+\mu^1_\alpha)} (\psi^1_\alpha - \psi^0_\alpha) = \Var_{\frac{1}{2}(\mu^0+\mu^1)} (\psi^1 - \psi^0).$$

\ref{it:bound-phi} We note that for $k \in \{0, 1\}$ and $x \in
\X$, $\mphi \leq \phi^k_\alpha(x) \leq \Mphi$. This is a simple
consequence of the definition of the Moreau-Yosida regularization
$\phi^k_\alpha$ as an infimum convolution. Indeed for any $x \in \X$, we
have on one hand:
\begin{align*}
  \phi^k_\alpha(x) &= \inf_{x' \in \X} \left( \phi^k(x') + \frac{1}{2 \alpha} \nr{x - x'}^2 \right) \\
  &\geq \inf_{x' \in \X} \phi^k(x')  + \inf_{x' \in \X}  \frac{1}{2 \alpha} \nr{x - x'}^2  \geq \mphi.
\end{align*}
On the other hand,
\begin{align*}
    \phi^k_\alpha(x) &= \inf_{x' \in \X} \left( \phi^k(x') + \frac{1}{2 \alpha} \nr{x - x'}^2 \right) \leq \phi^k(x) + \frac{1}{2 \alpha} \nr{x - x}^2 \leq \Mphi.
\end{align*}
We have all the desired properties \ref{it:cv-potentials}-\ref{it:bound-phi} but the potentials
$\phi^k_\alpha$ are not strongly convex and $\Class^2$ on $\Rsp^d$:
they are merely $\Class^{1,1}$. Moreover, the property \ref{it:support} does not
hold. These properties will be obtained thanks to a second regularization.

\emph{Second regularization.}  From now on, we fix some $\alpha > 0$,
and we denote $\Y_\alpha = B(0,R_\alpha)$ a closed ball that contains
the supports of $\mu^0_\alpha$ and $\mu^1_\alpha$. To construct the
regularization of $\phi^k_\alpha$ we will regularize the measures
$\mu^k_\alpha$ and solve an optimal transport problem. We first note
that it is straightforward, e.g. using a simple convolution and
truncation, to approximate the probability measures $\mu^k_\alpha$ on
$\Y_\alpha$ by smooth probability densities $\mu^k_{\alpha,\beta}$
supported on $\Y_\alpha$, bounded away from zero and
infinity on $\Y_\alpha$
and such that $\lim_{\beta \to 0} \Wass_2(\mu^k_{\alpha,\beta}, \mu^k_{\alpha}) = 0$. By
Caffarelli's regularity theory (e.g. Theorem 3.3 in
\cite{de2014monge}), the optimal transport map $T_{\alpha,\beta}$
between $\rho$ and $\mu^k_{\alpha,\beta}$ is the gradient of a
strongly convex potential $\phi^k_{\alpha,\beta}$ belonging to $\Class^2(\X)$ and is actually a diffeomorphism between $\X$ and $\Y_\alpha$. By Theorem 4.4 in \cite{yan2012extension}, the potential
$\phi^k_{\alpha,\beta}$ can be extended into a $\Class^2$ strongly
convex function on $\Rsp^d$. By stability of Kantorovich potentials
(Theorem 1.51 in \cite{santambrogio2015optimal}), taking a subsequence
if necessary, we can assume that $\phi^k_{\alpha,\beta}$ converges
uniformly to $\phi^k_\alpha$ on $\X$ as $\beta\to 0$.  Since $\nabla
\phi^k_{\alpha,\beta}$ sends $\rho$ to the measure
$\mu^k_{\alpha,\beta}$, which is supported on $B(0,R_\alpha)$, we get
$\nr{\nabla\phi^k_{\alpha,\beta}} \leq R_\alpha$. Moreover, since the convex
function $\phi^k_{\alpha,\beta}$ converges uniformly to
$\phi^k_\alpha$ as $\beta\to 0$, we get
$$ \hbox{ for a.e. } x\in \X, \lim_{\beta\to 0} \nabla \phi^k_{\alpha,\beta}(x) = \nabla \phi^k_\alpha(x).$$ 
This convergence result is also induced by the stability of optimal transport maps \cite[Corollary 5.21]{villani2008optimal}, since $\lim_{\beta \to 0} \Wass_2(\mu^k_{\alpha,\beta}, \mu^k_{\alpha}) = 0$ and $\nabla \phi^k_{\alpha, \beta}$ (resp. $\nabla \phi^k_\alpha$) is the optimal transport map between $\rho$ and $\mu^k_{\alpha,\beta}$ (resp. $\rho$ and $\mu^k_{\alpha}$). From these two properties we get as above the desired convergence properties: denoting with $\psi^k_{\alpha,\beta}$ the convex conjugate of $\phi^k_{\alpha,\beta}$ for $k \in \{0,1\}$,
\begin{enumerate}[label = (\roman*)]
\item  $\lim_{\beta\to 0} \sca{\psi^0_{\alpha,\beta} - \psi^1_{\alpha,\beta}}{\mu^1_{\alpha,\beta} - \mu^0_{\alpha,\beta}} = \sca{\psi^0_\alpha - \psi^1_\alpha}{\mu^1_\alpha - \mu^0_\alpha}$,
    \item  $\lim_{\beta\to 0}  \Var_{\frac{1}{2}(\mu^0_{\alpha,\beta}+\mu^1_{\alpha,\beta})}(\psi^1_{\alpha,\beta} - \psi^0_{\alpha,\beta}) = \Var_{\frac{1}{2}(\mu^0_\alpha+\mu^1_\alpha)}(\psi^1_\alpha - \psi^0_\alpha)$,
    \item $\mphi \leq \liminf_{\beta\to 0}\min_{\X} \phi^k_{\alpha,\beta}(x) \leq \limsup_{\beta\to 0}\max_{\X} \phi^k_{\alpha,\beta}(x) \leq \Mphi$.
\end{enumerate}
The sequence in the statement of the proposition is finally constructed using a diagonal argument: for $n \geq 1$, consider $\alpha = \alpha_n = \frac{1}{n}$ and denote
\begin{align*}
    \eps_n &= \abs{ \sca{\psi^0_{\alpha_n} - \psi^1_{\alpha_n}}{\mu^1_{\alpha_n} - \mu^0_{\alpha_n}} - \sca{\psi^0 - \psi^1}{\mu^1 - \mu^0} } \\
    &+ \abs{ \Var_{\frac{1}{2}(\mu^0_{\alpha_n}+\mu^1_{\alpha_n})}(\psi^1_{\alpha_n} - \psi^0_{\alpha_n}) - \Var_{\frac{1}{2}(\mu^0_n+\mu^1_n)}(\psi^1 - \psi^0) } \\
    &+ (m_\phi - \min_{\X} \min_k \phi^k_{\alpha_n})_+ + (\max_{\X} \max_k \phi^k_{\alpha_n} - M_\phi)_+,
\end{align*}
where $x_+ = \max(0, x)$. By the first regularization argument above, $\eps_n$ goes to $0$ as $n$ goes to infinity. Then for a given $n \geq 1$, choose $\beta = \beta_n$ small enough such that 
\begin{align*}
    \eps_n &\geq \abs{ \sca{\psi^0_{\alpha_n, \beta_n} - \psi^1_{\alpha_n, \beta_n}}{\mu^1_{\alpha_n, \beta_n} - \mu^0_{\alpha_n, \beta_n}} - \sca{\psi^0_{\alpha_n} - \psi^1_{\alpha_n}}{\mu^1_{\alpha_n} - \mu^0_{\alpha_n}} } \\
    &+ \abs{ \Var_{\frac{1}{2}(\mu^0_{\alpha_n, \beta_n}+\mu^1_{\alpha_n, \beta_n})}(\psi^1_{\alpha_n, \beta_n} - \psi^0_{\alpha_n, \beta_n}) - \Var_{\frac{1}{2}(\mu^0_{\alpha_n}+\mu^1_{\alpha_n})}(\psi^1_{\alpha_n} - \psi^0_{\alpha_n}) } \\
    &+ (\min_{\X} \min_k \phi^k_{\alpha_n} - \min_{\X} \min_k \phi^k_{\alpha_n, \beta_n})_+ + (\max_{\X} \max_k \phi^k_{\alpha_n, \beta_n} - \max_{\X} \max_k \phi^k_{\alpha_n})_+.
\end{align*}
By the second argument above, such a $\beta_n$ always exists. Then the sequences $(\phi^0_n)_{n\geq 1}, (\phi^1_n)_{n\geq 1}$defined for all $n \geq 1$ by $\phi^0_n = \phi^0_{\alpha_n, \beta_n}$ and $\phi^1_n = \phi^1_{\alpha_n, \beta_n}$ satisfy the regularity properties and convergence results of the statement.
\end{proof}

\begin{proposition} \label{prop:Poincare-smooth2} 
In addition to the assumptions of \cref{th:Poincare-ineq}, assume that $\X$ is a smooth and strongly convex set and that the density $\rho$ is smooth. Then,
\eqref{eq:Poincaré-psi} holds.
\end{proposition}

\begin{proof}
Let $\phi^0_n,\phi^1_n$ be the sequence of $\Class^2$ and strongly convex potentials constructed by \cref{prop:reg-psi}, converging respectively to $\phi^0$ and $\phi^1$, and such that
$\nabla\phi^0_n,\nabla\phi^1_n$ are diffeomorphisms from $\X$ to a ball $\Y_n$. By Proposition~\ref{prop:Poincare-ineq-smooth}, \eqref{eq:Poincaré-psi} holds for $\phi^0_n,\phi^1_n$:
$$\Var_{\mu^0_n+\mu^1_n} (\psi^1_n - \psi^0_n) \leq C_d \frac{ \Mrho^2}{\mrho^2} ({\Mphi}_n- {\mphi}_n) \sca{\psi^0_n - \psi^1_n}{\mu^1_n - \mu^0_n}.$$
By the claims (i)-(iv) in \cref{prop:reg-psi}, all the terms in this inequality converge as $n \to +\infty$ and establish \eqref{eq:Poincaré-psi} in the limit.
\end{proof}

\begin{proof}[Proof of \cref{th:Poincare-ineq}] 
Let $\X$ be a bounded convex set and assume that $\rho$ is a probability density satisfying $\mrho \leq \rho \leq \Mrho$. We extend $\rho$ by $\mrho$ outside of $\X$.
One can construct a sequence $\X_n$ of smooth and strongly convex sets included in $\X$ and converging to $\X$  in the Hausdorff sense as $n\to +\infty$ \cite[\S3.3]{schneider}. 
Let $K$ be a smooth, non-negative and compactly supported function, $K_n(x) = n^d K(nx)$ and define
$$ \rho_n = \frac{1}{Z_n} \restr{(\rho * K_n)}_{\X_n},  {\mrho}_n = \frac{\mrho}{Z_n}, {\Mrho}_n = \frac{\Mrho}{Z_n}, $$
where $Z_n$ is a constant ensuring that $\rho_n$ belongs to $\Prob(\X_n)$. We define $\mu^k_n = \nabla \phi^k_{\#}\rho_n$. Applying  
Proposition~\ref{prop:Poincare-smooth2} to $(\X_n,\rho_n)$ and $(\phi^0,\phi^1)$, we have:
\begin{equation}\label{eq:Poincaré-psin}
    \Var_{\mu^0_{n}+\mu^1_{n}} (\psi^1 - \psi^0) \leq C_d \frac{{\Mrho}_n^2}{{\mrho}_n^2} (\Mphi - \mphi) \sca{\psi^0 - \psi^1}{\mu^1_{n} - \mu^0_{n}}.
\end{equation}  By construction, $\lim_{n\to +\infty} Z_n = 1$ and $\rho_n$ converges to $\rho$ in $\L^1(\X)$. Thus up to subsequences, $\rho_n$ converges pointwise almost everywhere to $\rho$. Setting $f = \psi^0 - \psi^1$, we have
$$\sca{\psi^0 - \psi^1}{\mu^0_{n}} 
= \int_\X f(\nabla \phi^0) \rho_n(x) \dd x \xrightarrow{n\to +\infty} \int_\X f(\nabla \phi^0) \rho(x)\dd x =  \sca{\psi^0 - \psi^1}{\mu^0}. 
$$ 
The limit in the above equation is proven as in   Proposition~\ref{prop:reg-psi}, using that $f$ is Lipschitz, that $M_2(\mu_0)<+\infty$ and applying the dominated convergence theorem.
All the terms can be dealt with in a similar manner. Taking the limit $n\to +\infty$ in \eqref{eq:Poincaré-psin} gives the desired \eqref{eq:Poincaré-psi}. 
\end{proof}

\section{Stability of potentials}

\label{sec:stab-dual}

A direct consequence of the strong convexity estimate of Theorem \ref{th:Poincare-ineq} is a quantitative stability result on the dual potential $\psi_\mu$ with respect to the target measure $\mu$. This estimate on dual potentials is readily transferred to the Brenier (primal) potentials thanks to the next proposition. We refer to \cref{def:Brenier} for a definition of the potentials.

\begin{proposition} \label{prop:bound-lp-primal-dual} 
Let $\rho$ be a probability density over a compact convex set $\X$, and let $\phi^0,\phi^1$ be convex functions on $\X$. 
Denote $\psi^k$ the convex conjugate of $\phi^k$ and $\mu^k$ the image of $\rho$ under $\nabla \phi^k$. Then for any $p>0$,  
$$ \nr{\phi^1 - \phi^0}_{\L^p(\rho)} \leq \nr{ \psi^1 - \psi^0}_{\L^p(\mu^0 + \mu^1)}. $$
In particular, 
$$ \frac{1}{2} \Var_\rho(\phi^1 - \phi^0) \leq \Var_{\frac{1}{2}(\mu^0+\mu^1)}(\psi^1 - \psi^0).$$
\end{proposition}

Before proving this proposition, let us mention some consequences. The stability estimates resulting from Theorem \ref{th:Poincare-ineq} and this proposition are expressed in Corollary \ref{cor:stability-potentials} in terms of variance for the potentials and $1$-Wasserstein distance for the target measures. Assuming that one of the target measures is absolutely continuous with respect to the other, these estimates can also be expressed in term of $\chi^2$ or Kagan's divergence of the target measures. The $\chi^2$ divergence reduces to the $\chi^2$ test-statistic used for goodness of fit testing when the compared measures are finitely and commonly supported and one of them is observed empirically. Note that such divergence can be interpreted as the \emph{square of a divergence}, noting for instance that the total variation distance is only $\frac{1}{2}$-Hölder stable with respect to it \cite{comp_OT}.
\begin{corollary}[Stability of potentials]
\label{cor:stability-potentials}
Let $\rho$ be a probability density over a  compact convex set $\X$, satisfying $0 < \mrho \leq \rho \leq \Mrho$ and let $\mu^0, \mu^1 \in \Prob_2(\Rsp^d)$. For $k \in \{0, 1\}$, denote $\phi^k = \phi_{\mu^k}$ the Brenier potential between $\rho$ and $\mu^k$. Assume that $\phi^0,\phi^1$ satisfy \eqref{eq:hyp-phi} and denote 
$\psi^0$ and $\psi^1$  the convex conjugates of $\phi^0$ and $\phi^1$. Then,
\begin{equation*}
    \Var_{\rho}(\phi^1 - \phi^0) \leq 2\Var_{\frac{1}{2}(\mu^0+\mu^1)}(\psi^1 - \psi^0) \leq C_{d, \rho} \diam(\X)  (\Mphi - \mphi)\Wass_1(\mu^0, \mu^1) 
\end{equation*}
with $C_{d, \rho} = e(d+1)2^{d} \frac{\Mrho^2}{\mrho^2}$. Assuming additionally that $\mu^1$ is absolutely continuous w.r.t. $\mu^0$, then
\begin{equation*}
    \Var_{\rho}(\phi^1 - \phi^0) \leq 2 \Var_{\frac{1}{2}(\mu^0+\mu^1)}(\psi^1 - \psi^0) \leq C_{d, \rho}^2 (\Mphi - \mphi)^2 D_{\chi^2} (\mu^1 \vert \mu^0) 
\end{equation*}
where $D_{\chi^2} (\mu^1 \vert \mu^0)$ stands for the $\chi^2$ or Kagan's divergence from $\mu^1$ to $\mu^0$.

\end{corollary}

\begin{proof} 
Proposition \ref{prop:bound-lp-primal-dual} combined with Theorem \ref{th:Poincare-ineq} give the inequalities
\begin{align*}
       \Var_\rho(\phi^1 - \phi^0) &\leq 2 \Var_{\frac{1}{2}(\mu^0+\mu^1)} (\psi^1 - \psi^0) \\
       &\leq C_{d, \rho} (\Mphi - \mphi) \sca{\psi^0 - \psi^1}{\mu^1 - \mu^0}. 
  \end{align*}
The first estimate follows from Kantorovich-Rubinstein duality result:
for any $x \in \Rsp^d$ and $k \in \{0,1\}$, for any $g^k \in \partial \psi^k(x)$, one has $g^k \in \X$ so that $\psi^1 - \psi^0$ is $\diam(\X)$-Lipschitz continuous. Kantorovich-Rubinstein duality formula then ensures
\begin{align*}
       \sca{\psi^0 - \psi^1}{\mu^1 - \mu^0} \leq \diam(\X) \Wass_1(\mu^0, \mu^1).
\end{align*}
Now notice that if $\mu^1$ is absolutely continuous with respect to $\mu^0$, then we have for any constant $c \in \Rsp$:

\begin{align*}
    \sca{\psi^0 - \psi^1}{\mu^1 - \mu^0} &= \sca{\psi^0 - \psi^1 - c}{\mu^1 - \mu^0} \\
    &=  \int_{\Rsp^d} (\psi^0 - \psi^1 - c) (\frac{\dd \mu^1}{\dd \mu^0} - 1)\dd \mu^0 \\
    &\leq \left( \int_{\Rsp^d} (\psi^0 - \psi^1 - c)^2 \dd \mu^0 \right)^{1/2} \left( \int_{\Rsp^d} (\frac{\dd \mu^1}{\dd \mu^0} - 1)^2 \dd \mu^0 \right)^{1/2} \\
    &= \nr{\psi^0 - \psi^1 - c}_{\L^2(\mu^0)} D_{\chi^2} (\mu^1 \vert \mu^0)^{1/2}\\
    &\leq \sqrt{2}\nr{\psi^0 - \psi^1 - c}_{\L^2(\frac{1}{2}(\mu^0+\mu^1))} D_{\chi^2} (\mu^1 \vert \mu^0)^{1/2}.
\end{align*}
The second estimate comes after minimizing with respect to $c$ in the last inequality:
\begin{align*}
    \sca{\psi^0 - \psi^1}{\mu^1 - \mu^0} &\leq \sqrt{2} \Var_{\frac{1}{2}(\mu^0+\mu^1)}(\psi^1 - \psi^0)^{1/2} D_{\chi^2} (\mu^1 \vert \mu^0)^{1/2}. \qedhere
\end{align*} 
\end{proof}

\begin{proof}[Proof of Proposition \ref{prop:bound-lp-primal-dual}]
Let $A = \{x \in \X \mid \phi^1(x) \geq \phi^0(x) \}$ and let $x\in A$ where $\phi^1$ is differentiable. The Fenchel-Young inequality (and equality) give:
$$ \psi^0(\nabla \phi^1(x)) \geq \sca{x}{\nabla \phi^1(x)} - \phi^0(x) = \psi^1(\nabla \phi^1(x)) + \phi^1(x) - \phi^0(x), $$
which thus ensures that for  almost every  $x \in A$, 
$$ \psi^0(\nabla \phi^1(x)) - \psi^1(\nabla \phi^1(x)) \geq \phi^1(x) - \phi^0(x) \geq 0.$$
Similarly, for almost every $x \in \X\setminus A$, we have  
$$\psi^1(\nabla \phi^0(x)) - \psi^0(\nabla \phi^0(x)) \geq \phi^0(x) - \phi^1(x) \geq 0.$$
From this, we deduce the first statement of the proposition:
\begin{align*}
    \| \psi^1 - \psi^0&\|_{\L^p(\mu^0 + \mu^1)}^p = \int_\X\left( \abs{(\psi^1 - \psi^0)\circ(\nabla \phi^0)}^p +  \abs{(\psi^1 - \psi^0)\circ (\nabla \phi^1)}^p\right)\dd\rho \\
    &\geq \int_{\X\setminus A} \big(\psi^1(\nabla \phi^0) - \psi^0(\nabla \phi^0)\big)^p\dd\rho + \int_A \big(\psi^0(\nabla \phi^1) - \psi^1(\nabla \phi^1)\big)^p\dd\rho \\
    &\geq \int_{\X\setminus A} \big(\phi^0 - \phi^1\big)^p\dd\rho + \int_A \big(\phi^1 - \phi^0\big)^p\dd\rho = \nr{\phi^1 - \phi^0}^p_{\L^p(\rho)}.
\end{align*}
Let $c \in \Rsp$. Having established the previous inequality for any convex functions $\phi^0, \phi^1$ on $\X$, we may replace $\phi^0$ with $\phi^0 - c$ in this inequality, and consequently replace $\psi^0$ with $(\phi^0 - c)^* = \psi^0 + c$. This yields thus for any $c \in \Rsp$,
$$ \nr{\phi^1 - \phi^0 + c}_{\L^p(\rho)} \leq \nr{ \psi^1 - \psi^0 - c}_{\L^p(\mu^0 + \mu^1)}. $$
Taking $c$ that achieves the minimum on the right-hand side, for $p=2$, we get
\begin{align*}
 \frac{1}{2} \Var_\rho(\phi^1 - \phi^0) &\leq \frac{1}{2} \nr{\phi^1 - \phi^0 + c}^2_{\L^2(\rho)} \\
 &\leq  \nr{ \psi^1 - \psi^0 - c}^2_{\L^2(\frac{1}{2}(\mu^0+\mu^1))} = \Var_{\frac{1}{2}(\mu^0+\mu^1)}(\psi^1 - \psi^0). \qedhere      
\end{align*}
\end{proof}

All the stability estimates that have been established so far involve the oscillation of the Brenier potentials $\Mphi-\mphi$. It is then natural to wonder 
under what assumption on a measure $\mu \in \Prob_2(\Rsp^d)$  can we control this oscillation. The next proposition, found in \cite{berman2013real}, shows that a sufficient condition is that $\mu$ admits a finite moment of order $p>d$. This assumption seems nearly tight : Remark~\ref{rem:morrey} below shows that there exists a measure $\mu$ such that $M_p(\mu)<+\infty$ with $p<d$, whose associated Brenier potential is unbounded.

\begin{proposition}[Proposition 2.22 in \cite{berman2013real}]
\label{prop:morrey-phi}
 Let $\rho$ be a probability density over a  compact convex set $\X$, satisfying $0 < \mrho \leq \rho \leq \Mrho$ and let $\mu \in \Prob_2(\Rsp^d)$. Denote $\phi$ the Brenier potential for the quadratic optimal transport between $\rho$ and $\mu$. Assume that there exists $p > d$ and $M_p < +\infty$ such that
$$ M_p(\mu) = \int_{\Rsp^d} \nr{y}^p \dd \mu(y) \leq M_p. $$
Then $\phi$ is Hölder continuous and verifies for all $x, x' \in \X$:
\begin{align*}
    \abs{\phi(x) - \phi(x')} \leq C_{d, p, \X} \left(\frac{M_p}{\mrho}\right)^{1/p}  \nr{x-x'}^{1 - \frac{d}{p}}.
\end{align*} 
In particular, there exists 
$\mphi, \Mphi \in \Rsp$ that can be chosen such that for any $x \in \X$, $\mphi \leq \phi(x) \leq \Mphi$ and such that
$$ \Mphi - \mphi \leq C_{d, p, \X} \left(\frac{M_p}{\mrho}\right)^{1/p} \diam(\X)^{1 - \frac{d}{p}}.$$
\end{proposition} 

Corollary \ref{cor:stability-potentials} and Proposition \ref{prop:morrey-phi} together imply the following.

\begin{corollary}[Stability with enough moments] \label{cor:stab-pot-mp}
 Let $\rho$ be a probability density over a compact convex set $\X$, satisfying $0 < \mrho \leq \rho \leq \Mrho$. For any $\mu \in \Prob_2(\Rsp^d)$, denote $\phi_\mu$ the Brenier potential for the optimal transport between $\rho$ and $\mu$. Let $p > d$. Then the restriction of the mapping $\mu \mapsto \phi_\mu$ to the set of probability measures with bounded $p$-th moment is $1/2$-Hölder with respect to the
$\Wass_1$ distance. More precisely, if $\max(M_p(\mu^0),M_p(\mu^1)) \leq M_p < +\infty,$ then
$$ 
\nr{\phi_{\mu^1} - \phi_{\mu^0}}_{\L^2(\rho)} \leq C_{d, p, \X, \rho, M_p} \Wass_1(\mu^0, \mu^1)^{1/2}.
$$
\end{corollary}

\begin{remark}
\label{rk:sub-exponential}
A large class of probability distributions admit a finite moment of order $p>d$. For instance, sub-exponential measures, which encompass most of the commonly used heavy-tailed distributions fall into this class. We say that a measure $\mu \in \mathcal{P}\left(\Rsp^{d}\right)$ is sub-exponential with variance proxy $\sigma^{2}$ for $\sigma > 0$ if it has zero mean  and if for all $r>0$, $$ \mu(\{x\in\Rsp^d\mid\nr{x}\geq r\}) \leq 2 e^{-2r/\sigma}.$$ We refer to  Proposition 2.7.1 in \cite{vershynin_2018} for equivalent characterization. The moments of such a measure  are all bounded, and more precisely, 
$$ M_p(\mu) \leq 2 p! \left(\frac{\sigma}{2}\right)^p .$$
\end{remark}

We report the proof of Proposition \ref{prop:morrey-phi} from \cite{berman2013real} for completeness. 

\begin{proof}[Proof of Proposition \ref{prop:morrey-phi}]
The gradient $\nabla \phi$ corresponds to the optimal transport map between $\rho$ and $\mu$. Using that $\mu$ is the image of $\rho$ under $\nabla \phi$, the moment assumption  gives,
\begin{align*}
    \nr{\nabla \phi}_{\L^p(\X)}^p &= \int_\X \nr{\nabla \phi (x)}^p \dd x \leq \frac{1}{\mrho} \int_\X \nr{\nabla \phi (x)}^p \dd \rho(x) \leq \frac{M_p}{\mrho}.
\end{align*}
We can add a constant to $\phi$ so that $\int_\X \phi(x) \dd x = 0$ without changing its modulus of continuity. The Poincaré-Wirtinger inequality then ensures that 
$\nr{\phi}_{\L^p(\X)} \leq C_{p, \X} \nr{\nabla \phi}_{\L^p(\X)}.$
In particular, the potential $\phi$ belongs to the Sobolev space $W^{1,p}(\X)$.
Morrey's inequality (Theorem 11.34 and Theorem 12.15 in \cite{Leoni}) ensures that $\phi$ is $(1-\frac{d}{p})$-Hölder  and that there exists a constant depending only on $d, p$ and $\X$  such that

\begin{align*}
    \forall x\neq x' \in \X, \quad \frac{ \abs{\phi(x) - \phi(x')} }{ \nr{x-x'}^{1 - \frac{d}{p}} } &\leq C_{d, p, \X} \nr{\phi}_{W^{1,p}(\X)} 
    \leq C_{d, p, \X} \left(\frac{M_p}{\mrho} \right)^{1/p}. \qedhere
\end{align*}
\end{proof}

\begin{remark}[Morrey's inequality for convex functions] \label{rem:morrey}
Since the Brenier potentials $\phi$ are convex, one may wonder whether Morrey's inequality and the resulting Sobolev embedding can be improved when restrictected to the class of convex functions. However, one can show that for $\X = [0, 1]^d$ and $p < d$, for $\alpha \in \left(0, \frac{d}{p} - 1\right)$, the potential
$$
\phi :\left\{
    \begin{array}{ll}
        \X \to \Rsp \\
        (x_1, \dots, x_d) \mapsto (x_1 + \dots + x_d)^{-\alpha}
    \end{array}
\right.
$$
is convex, belongs to $W^{1, p}(\X)$, but obviously neither Hölder continuous nor even bounded. In other words, assuming that $M_p(\mu)<+\infty$ for $p<d$ does not guarantee that the Brenier potential from $\rho$ to $\mu$ is $\alpha$-Hölder, or even bounded.

\end{remark}

\section{Stability of optimal transport maps}
\label{sec:stab-ot-maps}

In this section, we derive quantitative stability estimates on optimal transport maps with respect to the target measures from the stability estimates on Brenier potentials given in the preceding section. This derivation relies on a Gagliardo–Nirenberg type inequality on the difference of convex  functions, which is reported here but will be proven in Section \ref{sec:ineg-convex-functions}.

\begin{proposition}
\label{prop:bound-diff-convex-fns}
Let $K$ be a compact domain of $\Rsp^d$ with rectifiable boundary and let $u, v: K \to \Rsp$ be two $L$-Lipschitz functions on $K$ that are convex on any
segment included in $K$. Then there exists a constant $C_{d}$ depending only on $d$ such that
\begin{equation*} \label{eq:gn}
    \nr{\nabla u - \nabla v}_{\L^2(K, \Rsp^d)}^2 \leq C_{d} \mathcal{H}^{d-1}(\partial K)^{2/3} L^{4/3}  \nr{u - v}^{2/3}_{\L^2(K)},
\end{equation*}
where $\mathcal{H}^{d-1}$ denotes the $(d-1)$-dimensional Hausdorff measure.
\end{proposition}

With this proposition at hand, the stability result for Brenier potentials can readily be transferred to stability of the corresponding optimal transport maps --  that is, to their gradient -- at least when the target measures are  compactly supported. Indeed, Proposition \ref{prop:bound-diff-convex-fns} together with Corollary \ref{cor:stability-potentials}  directly imply:

\begin{theorem}[Stability of the Brenier map, compact case]
\label{th:stability-ot-maps-compact-case}
Let $\X,\Y$ be compact subsets of $\Rsp^d$ with $\X$ convex, let $\rho$ be a probability density over $\X$ bounded from above and below by positive constants 
and let $\mu^0,\mu^1 \in \Prob(\Y)$. Denoting $T_{\mu^k}$ the Brenier map from $\rho$ to $\mu^k$, we have
$$ \Wass_2(\mu^0,\mu^1) \leq \nr{T_{\mu^0} - T_{\mu^1}}_{\L^2(\rho, \Rsp^d)} \leq C_{d,\rho,\X,\Y}  \Wass_1(\mu^0, \mu^1)^{\frac{1}{6}}.$$
In particular, the embedding $\mu \in \Prob_2(\Y) \to T_\mu\in \L^2(\rho,\Rsp^d)$ is bi-Hölder continuous.
\end{theorem} 

\begin{remark}[bi-Hölder embedding via potentials]
The previous theorem and Proposition \ref{prop:bound-diff-convex-fns} together with Corollary \ref{cor:stability-potentials} also ensure the following
bi-Hölder behavior for the Brenier potentials (with zero mean against $\rho$ on $\X$):
\begin{equation*}
  \forall \mu^0,\mu^1\in\Prob(\Y), \quad \Wass_2(\mu^0, \mu^1)^3 \lesssim \nr{\phi^1 - \phi^0}_{\L^2(\rho)} \lesssim \Wass_1(\mu^0, \mu^1)^{\frac{1}{2}},
\end{equation*}
where the $\lesssim$ notation hides multiplicative constants depending on $d,\rho,\X,\Y$.
\end{remark}

We now phrase a similar stability result for probability measures whose Brenier potential is Hölder continuous and that admit a bounded fourth order moment. 
This includes a large class of probability measures, as noticed in Proposition \ref{prop:morrey-phi} and Remark \ref{rk:sub-exponential}.

\begin{theorem}[Stability of the Brenier map]
\label{th:stability-ot-maps}
Let $\rho$ be a probability density over a  compact convex set $\X \subset \Rsp^d$, satisfying $0 < \mrho \leq \rho \leq \Mrho$.
Let $\mu^0, \mu^1 \in \Prob_2(\Rsp^d)$ and denote $\phi^0, \phi^1$ the Brenier potentials for the quadratic optimal transport between $\rho$ and $\mu^0, \mu^1$ respectively. Assume that there exists $M_\alpha > 0$ and $\alpha \in (0, 1)$ such that for all $x, x' \in \X$ and $k \in \{0,1\}$,
$$ \abs{\phi^k(x) - \phi^k(x')} \leq M_\alpha \nr{x - x'}^\alpha.$$
Assume that there exists $0 < M < +\infty$ such that for $ k \in \{0, 1\}$, $ M_4(\mu^k) \leq M.$
Then 
\begin{equation}
    \label{eq:stab-ot-maps}
    \Wass_2(\mu^0, \mu^1) \leq \nr{\nabla \phi^1 - \nabla \phi^0}_{\L^2(\rho, \Rsp^d)} \leq C_{d,\rho,\X,\alpha,M_\alpha,M} \Wass_1(\mu^0, \mu^1)^{\frac{1}{2(11-8\alpha)}}.
\end{equation}
\end{theorem}

\begin{remark}
The assumption $M_4(\mu^k) < + \infty$ comes from a use of the Cauchy-Schwarz inequality in the proof of Theorem \ref{th:stability-ot-maps}. However, one could use Hölder's inequality instead, under different moment assumption and show that for any $q \geq 1$, assuming that $M_{2q}(\mu^k) \leq M_{2q} <+ \infty$ for $k \in \{0, 1\}$, one has
$$ \nr{\nabla \phi^1 - \nabla \phi^0}_{\L^2(\rho)} \leq C_{d, \rho, \X, M_\alpha, \alpha, M_{2q}} \Wass_1(\mu^0, \mu^1)^{\frac{q-1}{2(q(7-4\alpha) - 3)}}. $$
Since the exponent is an increasing function of $q$, a stronger stability can be obtained at the cost of stronger moment assumptions.
\end{remark}

Theorem \ref{th:stability-ot-maps} and Proposition \ref{prop:morrey-phi} directly imply the following.

\begin{corollary}[Stability with enough moments]
\label{cor:embedding-BWp-Mp}
Let $\rho$ be a probability density over a  compact convex set $\X \subset \Rsp^d$, satisfying $0 < \mrho \leq \rho \leq \Mrho$.
For $\mu \in \Prob_2(\Rsp^d)$, denote $\nabla \phi_\mu$ the optimal transport map for the quadratic optimal transport between $\rho$ and $\mu$. Let $p \in \Rsp$ and assume $p \geq 4$ and $p>d$. Then, the map $\mu\mapsto T_\mu$ is Hölder when restricted to the set of probability measures with bounded $p$-th moment. More precisely, if $\max(M_p(\mu^0),M_p(\mu^1)) \leq M_p < +\infty,$ then
$$  \Wass_2(\mu^0, \mu^1) \leq \nr{\nabla \phi_{\mu^1} - \nabla \phi_{\mu^0}}_{\L^2(\rho, \Rsp^d)} \leq C_{d, p, \X, \rho, M_p} \Wass_1(\mu^0, \mu^1)^{\frac{p}{6p + 16d}}.$$
\end{corollary}

To prove Theorem~\ref{th:stability-ot-maps}, we first show that whenever a Brenier potential defined on the compact and convex set $\X$ is Hölder continuous, it is possible to control its Lipschitz constant on erosions of $\X$. We recall that for $\eta>0$, the $\eta$-erosion of $\X$, denoted $\X_{-\eta}$, corresponds to the set of points of $\X$ that are at least at a distance $\eta$ from $\partial \X$.  
The proof of this proposition is inspired by Proposition~3.3 in \cite{klartag2014logarithmically}.

\begin{proposition}[Lipschitz behavior on erosion]
\label{prop:bound-gradient-phi-t}
Let $\rho$ be a probability density over a  compact convex set $\X \subset \Rsp^d$, satisfying $0 < \mrho \leq \rho \leq \Mrho$.
Let $\mu \in \Prob_2(\Rsp^d)$ and denote $\phi$ the Brenier potential for the quadratic optimal transport between $\rho$ and $\mu$. Assume that there exists $M_\alpha > 0$ and $\alpha \in (0, 1)$ such that for all $x, x' \in \X$,
$$ \abs{\phi(x) - \phi(x')} \leq M_\alpha \nr{x - x'}^\alpha.$$
Then, $\phi$ is $R$-Lipschitz on the erosion $\X_{-\eta_R}$ with $\eta_R = \left( \frac{M_\alpha}{R} \right)^{\frac{1}{1 - \alpha}}$.
\end{proposition}

\begin{proof}
Let $x \in \X$ be such that $d(x, \partial \X) \geq \eta_R$, and let $g \in \partial \phi(x)$. We will show that $\nr{g} \leq R$, thus implying the statement. Denoting $\psi = (\phi)^*$, the Fenchel-Young equality and inequality ensures that
$$
\begin{cases}
    \psi(g) = \sca{g}{x} -  \phi(x), \\
    \psi(g) \geq \sca{g}{x'} - \phi(x') &\hbox{ for all } x'\in \X.
\end{cases}
$$
Putting these equations together, we get that for any $x' \in \X$,
\begin{align}
    \label{eq:psi-bound-holder}
    \sca{g}{x' - x} \leq \phi(x') - \phi(x) \leq M_\alpha \nr{x'-x}^\alpha, 
\end{align}
where we used  the Hölder continuity assumption on $\phi$. We now choose $x'$ to be the unique point in the intersection between the ray
$x + \Rsp^+ g$ and $\partial \X$, so that $\sca{g}{x'-x} =  \nr{x-x'}\nr{g}$ and in \eqref{eq:psi-bound-holder},
\begin{equation*}
    \nr{g} \leq \frac{M_\alpha}{\nr{x - x'}^{1-\alpha}}. 
\end{equation*}
Now using $\nr{x'-x} \geq d(x,\partial \X) \geq \eta_R$ in this last inequality yields $\nr{g} \leq R$.

\end{proof}

Proposition \ref{prop:bound-gradient-phi-t} allows to control the Lipschitz constant of the restriction $\phi^k$ to $\X_{-\eta}$ assuming that $\phi^k$ is $\alpha$-Hölder continuous. Combining it with the inequality of Proposition \ref{prop:bound-diff-convex-fns}, we get a stability estimate for the restriction of the transport map to $\X_{-\eta}$. To conclude the proof of the theorem, we will rely on an upper bound on the volume of the symetric difference betwen $\X$ and its erosion $\X_{-\eta}$ given in the next proposition.

\begin{proposition}[Volume of boundary slices]
\label{prop:volume-boundary-slice-cvx}
Let $\X \subset \Rsp^d$ be a compact convex set containing the origin, and denote $r_\X>0$ and $R_\X>0$ the largest and smallest radii
such that $B(0, r_\X) \subseteq \X \subseteq B(0, R_\X)$.
Then, for all $\eta  \geq 0$, 
$$ \vol^d(\X \backslash \X_{-\eta}) \leq 2 S_{d-1} (R_\X + r_\X)^{d-1} \frac{R_\X}{r_\X}  \eta,$$
where $S_{d-1}$ denotes the surface area of the $(d-1)$-dimensional unit sphere. 
\end{proposition}

We  quote a lemma extracted from \cite{Matheron2005THESF} that allows to control the volume of the  difference between a convex $\X$ and its $\eta$-erosion $X_{-\eta}$ using the volume of $\eta$-dilation of $\X$, denoted $\X_{+\eta} = \{ x \in \Rsp^d \mid  d(x, \X) \leq \eta\}. $

\begin{lemma}[Lemma 1 in \cite{Matheron2005THESF}]
\label{lemma:dilated-erosion}
For all $\eta \leq r_\X$,
$\vol^d(\X \backslash \X_{-\eta}) \leq \vol^d(\X_{+\eta} \backslash \X).$
\end{lemma}

This lemma, together with Steiner's formula already implies that $\vol^d(\X\setminus \X_{-\eta})$ grows linearly in $\eta$ for small values of $\eta$. We provide a direct proof below. 

\begin{proof}[Proof of \cref{prop:volume-boundary-slice-cvx}]
This result is proven using the radial function of $\X$, $\Lambda_\X(x) = \max \{ \lambda \geq 0 \vert \lambda x \in \X \}.$
Since $x\in\Sph^{d-1} \mapsto \Lambda_\X(x) x$ is a radial parametrization of $\partial \X$,  we have:
\begin{align*}
    \vol^d(\X) = \int_{\X} 1 \dd x =  \int_{\mathcal{S}^{d-1}} \int_0^{\Lambda_\X(u)} r^{d-1} \dd r \dd u = \frac{1}{d} \int_{\mathcal{S}^{d-1}} \Lambda_\X(u)^d \dd u.
\end{align*}
Combined with Lemma \ref{lemma:dilated-erosion}, this implies that for any $0 \leq \eta \leq r_\X$,
\begin{align*}
    \vol^d(\X \backslash \X_{-\eta}) &\leq \vol^d(\X_{+\eta} \backslash \X)
    = \frac{1}{d} \int_{\mathcal{S}^{d-1}} \left( \Lambda_{\X_{+\eta}}(u)^d - \Lambda_\X(u)^d \right) \dd u \\
    &= \frac{1}{d} \int_{\mathcal{S}^{d-1}} \left( \Lambda_{\X_{+\eta}}(u) - \Lambda_\X(u) \right) \left( \sum_{k=0}^{d-1} \Lambda_{\X_{+\eta}}(u)^{d-1-k} \Lambda_\X(u)^{k} \right) \dd u \\
    &\leq \frac{1}{d} \int_{\mathcal{S}^{d-1}} \left( \Lambda_{\X_{+\eta}}(u) - \Lambda_\X(u) \right) d \cdot(R_\X + r_\X)^{d-1} \dd u.
\end{align*}
Using the inclusions $B(0, r_\X) \subseteq \X \subseteq B(0, R_\X)$, one can prove that for any $\eta>0$ and  for any unit vector $u$, $$0 \leq \Lambda_{\X_{+\eta}}(u) - \Lambda_\X(u) \leq \frac{(r_\X^2 + R_\X^2)^{1/2}}{r_\X} \eta \leq \frac{2 R_\X}{r_\X} \eta.$$ This can be seen from the \emph{worst case} where $\X$ is an \emph{ice cream cone} made from the convex hull of $B(0, r_\X)$ and a point at distance $R_\X$ of the origin. 
This finally gives, for $\eta \in[0,r_\X],$
\begin{align*}
    \vol^d(\X \backslash \X_{-\eta}) \leq  \int_{\mathcal{S}^{d-1}} \frac{2R_\X}{r_\X} \eta (R_\X + r_\X)^{d-1} \dd u 
    = 2 S_{d-1} (R_\X + r_\X)^{d-1} \frac{R_\X}{r_\X}  \eta. 
\end{align*} 
One can easily check that in the case $\eta \geq r_\X$ the inequality also holds.

\end{proof}

\begin{proof}[Proof of Theorem \ref{th:stability-ot-maps}] In the following, the $\lesssim$ notation hides multiplicative constants that might depend on $d,\rho,\X,\alpha, M_\alpha, M$.
We get the left  inequality of \eqref{eq:stab-ot-maps} by recalling that
$$ \Wass_2(\mu^0, \mu^1)^2 = \min_{\gamma \in \Pi(\mu^0, \mu^1)} \int_{\Rsp^d \times \Rsp^d} \nr{x - y}^2 \dd \gamma(x,y), $$
and by noticing that the optimal transport maps $\nabla \phi^0, \nabla \phi^1$ between $\rho$ and $\mu^0, \mu^1$ yield an admissible coupling $\gamma^{0,1} := (\nabla \phi^0, \nabla \phi^1)_\# \rho \in \Pi(\mu^0, \mu^1)$, which leads to:
$$ \Wass_2(\mu^0, \mu^1)^2 \leq \int_{\Rsp^d \times \Rsp^d} \nr{x - y}^2 \dd \gamma^{0,1}(x,y) = \int_\X \nr{\nabla \phi^1 - \nabla \phi^0}^2 \dd \rho. $$
We  now prove the right  inequality of \eqref{eq:stab-ot-maps}. We recall that $\eta_R = \left( \frac{M_\alpha}{R} \right)^\frac{1}{1 - \alpha}$. Then, denoting $\rho_R$ the restriction of $\rho$ to $\X_{-\eta_R}$ and $\rho_R^\perp = \rho - \rho_R$,
\begin{samepage}
\begin{align*}
 \nr{\nabla \phi^1 - \nabla \phi^0}_{\L^2(\rho, \Rsp^d)}^2 &= \nr{\nabla \phi^1 - \nabla \phi^0 }_{\L^2(\rho_R,  \Rsp^d)}^2  + \nr{\nabla \phi^1 - \nabla \phi^0 }_{\L^2(\rho_R^\perp,  \Rsp^d)}^2.
\end{align*}
\end{samepage}
On $\X_{-\eta_R}$, Proposition \ref{prop:bound-gradient-phi-t} ensures that $\nr{\nabla \phi^k} \leq R$ for $k \in \{0, 1\}$. This fact thus ensures with Proposition \ref{prop:bound-diff-convex-fns} that for any $c \in \Rsp$:
\begin{align*}
     \nr{\nabla \phi^1 - \nabla \phi^0}_{\L^2(\rho_R,  \Rsp^d)}^2 &\lesssim R^{4/3} \nr{\phi^1 - \phi^0 - c}_{\L^2(\rho_R)}^2. 
\end{align*}
Note that we used the inequality $\Haus^{d-1}(\partial \X_{-\eta_R}) \leq \Haus^{d-1}(\partial \X)$ obtained from the inclusion of the convex set $\X_{-\eta_R}$ into $\X$, where the convexity of  $\X_{-\eta_R}$ is visible from $\X_{-\eta_R} = \bigcap_{\nr{e} = \eta_R} (\X - e)$. Minimizing over $c$ in the last inequality thus ensures
\begin{equation}
\label{eq:stability-X(R)}
    \nr{\nabla \phi^1 - \nabla \phi^0}_{\L^2(\rho_R,  \Rsp^d)}^2 \lesssim R^{4/3} \Var_\rho (\phi^1 - \phi^0)^{1/3} \lesssim R^{4/3} \Wass_1(\mu^0, \mu^1)^{1/3},
\end{equation}
where we used Corollary \ref{cor:stability-potentials} to get the second inequality. 
On the other hand, notice that
\begin{align*}
    \nr{\nabla \phi^1 - \nabla \phi^0}_{\L^2(\rho_R^\perp,  \Rsp^d)}^2 \leq 2 \nr{\nabla \phi^1}_{\L^2(\rho_R^\perp,  \Rsp^d)}^2 + 2\nr{\nabla \phi^0}_{\L^2(\rho_R^\perp,  \Rsp^d)}^2.
\end{align*}
By the Cauchy-Schwartz inequality we have for $k \in \{0, 1\}$
\begin{align*}
    \nr{\nabla \phi^k}_{\L^2(\rho_R^\perp,  \Rsp^d)}^2 &= \int_{ \X \setminus \X_{-\eta_R}} \nr{\nabla \phi^k}^2 \dd \rho \\
    &\leq \left( \int_{ \X \setminus \X_{-\eta_R}} \nr{\nabla \phi^k}^4 \dd \rho \right)^{1/2} \left( \int_{ \X \setminus \X_{-\eta_R}} 1^2 \dd \rho \right)^{1/2} \\
    &\lesssim M_4(\mu^k)^{1/2} \vol^d(\X \setminus \X_{-\eta_R})^{1/2}.
\end{align*}
\cref{prop:volume-boundary-slice-cvx} ensures that for any $R \geq 0$, we have
$$ \vol^d(\X \setminus \X_{-\eta_R}) \lesssim \eta_R = \left( \frac{M_\alpha}{R} \right)^{1/(1-\alpha)}. $$
This gives thus the estimation
\begin{equation}
\label{eq:stability-X-X(R)}
    \nr{\nabla \phi^1 - \nabla \phi^0}_{\L^2(\rho_R^\perp,  \Rsp^d)}^2 \lesssim R^{-{1/2(1-\alpha)}}
\end{equation}
Estimations \eqref{eq:stability-X(R)} and \eqref{eq:stability-X-X(R)} thus give for $R \geq 0$ 
\begin{equation}
\label{eq:stability-arbitrage-2}
    \nr{\nabla \phi^1 - \nabla \phi^0}_{\L^2(\rho, \Rsp^d)}^2 \lesssim R^{4/3} \Wass_1(\mu^0, \mu^1)^{1/3} + R^{-{1/2(1-\alpha)}}.
\end{equation}
Solving for $R^{4/3} \Wass_1(\mu^0, \mu^1)^{1/3} = R^{-{1/2(1-\alpha)}}$ yields $R =\Wass_1(\mu^0, \mu^1)^{\frac{-2(1 - \alpha)}{11 - 8\alpha}}$. Injecting this value of $R$ in \eqref{eq:stability-arbitrage-2} yields the desired estimate. 

\end{proof}

We finally prove that if the target measures $\mu^0, \mu^1$ are supported on a compact set $\Y \subset \Rsp^d$, if they are absolutely continuous and if their  densities are bounded away from zero and infinity, then the Hölder exponents can be slightly improved.

\begin{corollary}
\label{cor:stab-target-ac-compact-set}
Let $\X,\Y$ be compact subsets of $\Rsp^d$, and assume that $\X$ is convex and that $\Y$ has a rectifiable boundary. Let $\rho$ be a probability density over $\X$ satisfying $0 < \mrho \leq \rho \leq \Mrho < +\infty$ and let $\mu^0,\mu^1$ be probability densities over $\Y$ satisfying 
$$\forall k \in \{0, 1\}, \quad 0 < c_\mu \leq \mu^k \leq C_\mu < +\infty. $$
Then, if $\phi^k$ (resp. $T^k$) is the Brenier potential (resp. Brenier map) from $\rho$ to $\mu^k$, we have
\begin{gather*}
    \Wass_2(\mu^0, \mu^1)^6 \lesssim \Var_\rho(\phi^1 - \phi^0) \leq 2 \Var_{\frac{1}{2}(\mu^0+\mu^1)}(\psi^1 - \psi^0) \lesssim \Wass_2(\mu^0, \mu^1)^{\frac{6}{5}}, \\
    \Wass_2(\mu^0,\mu^1) \leq \nr{T^1 - T^0}_{\L^2(\rho, \Rsp^d)} \lesssim \Wass_2(\mu^0, \mu^1)^{\frac{1}{5}},
\end{gather*}
where the $\lesssim$ notation hides multiplicative constants depending on $d,\rho,\X,\Y, c_\mu$ and $C_\mu$.
\end{corollary}   
This corollary will be a consequence of the following lemma from \cite{bertr2010macroscopic}, which we will use as a replacement of the Kantorovich-Rubinstein inequality.

\begin{lemma}[Lemma 3.5 in \cite{bertr2010macroscopic}]
\label{lemma:ineq-cs-grad-W2}
Assume that $\mu^0$ and $\mu^1$ are absolutely continuous measures on the compact $\Y$, whose densities are bounded by a common constant $C_\mu$. Then, for any function $f \in H^1(\Y)$, we have the following inequality:
$$ \int_\Y f \dd(\mu^1 - \mu^0) \leq \sqrt{C_\mu} \| \nabla f \|_{\L^2(\Y)} \Wass_2(\mu^0, \mu^1). $$
\end{lemma}
\begin{proof}[Proof of Corollary \ref{cor:stab-target-ac-compact-set}]
Because $\Y$ is compact, the Brenier potentials $\phi^0, \phi^1$ are $R_\Y$-Lipschitz continuous for any $R_\Y \in \Rsp_+$ such that $\Y \subset B(0, R_\Y)$. One can thus find $m_\phi, M_\phi \in \Rsp$ such that for $k \in \{0,1\}, m_\phi \leq \phi^k \leq M_\phi$ on $\X$ and $M_\phi - m_\phi \leq R_\Y \diam(\X)$. Setting $\psi^0 = (\phi^0)^*, \psi^1 = (\phi^1)^*$, we thus have from \eqref{eq:Poincaré-psi} in Theorem \ref{th:Poincare-ineq}:
\begin{equation}
\label{eq:Poincaré-psi-bis}
\Var_{\frac{1}{2}(\mu^0+\mu^1)} (\psi^1 - \psi^0) \lesssim \sca{\psi^0 - \psi^1}{\mu^1 - \mu^0}.
\end{equation}
For $c \in \Rsp$ such that $\nr{\psi^1 - \psi^0 - c}_{\L^2(\frac{1}{2}(\mu^0+\mu^1))}^2 = \Var_{\frac{1}{2}(\mu^0+\mu^1)}(\psi^1 - \psi^0)$, estimation \eqref{eq:Poincaré-psi-bis} and Lemma \ref{lemma:ineq-cs-grad-W2} ensure that:
\begin{equation}
    \label{eq:step-1}
    \nr{\psi^1 - \psi^0 - c}_{\L^2(\frac{1}{2}(\mu^0+\mu^1))}^2 \lesssim \nr{\nabla \psi^1 - \nabla \psi^0}_{\L^2(\Y)} \Wass_2(\mu^0, \mu^1).
\end{equation}
But Proposition \ref{prop:bound-diff-convex-fns} applied to the convex and Lipschitz functions $\psi^0 + c, \psi^1$ ensures that
$$ \nr{\nabla \psi^1 - \nabla \psi^0}_{\L^2(\Y)} \lesssim \nr{\psi^1 - \psi^0 - c}_{\L^2(\frac{1}{2}(\mu^0+\mu^1))}^{1/3}.$$
Injecting this estimation into \eqref{eq:step-1} yields

\begin{equation*}
    \nr{\psi^1 - \psi^0 - c}_{\L^2(\frac{1}{2}(\mu^0+\mu^1))}^{2} \lesssim  \Wass_2(\mu^0, \mu^1)^{6/5}.
\end{equation*}
This gives thus with Proposition \ref{prop:bound-lp-primal-dual}
\begin{equation*}
\Var_\rho(\phi^1 - \phi^0) \leq 2 \Var_{\frac{1}{2}(\mu^0+\mu^1)} (\psi^1 - \psi^0) \lesssim \Wass_2(\mu^0, \mu^1)^{6/5}.
\end{equation*}
Finally, a last use of Proposition \ref{prop:bound-diff-convex-fns} also ensures that under these assumptions on the targets $\mu^0, \mu^1$ we have
\begin{align*} 
    \Wass_2(\mu^0,\mu^1) \leq  \|\nabla \phi^1 - \nabla \phi^0 \|_{\L^2(\rho, \Rsp^d)} &\lesssim \Var_\rho(\phi^1 - \phi^0)^{\frac{1}{6}} \lesssim \Wass_2(\mu^0, \mu^1)^{\frac{1}{5}}. \qedhere 
\end{align*}
\end{proof}

\section{Gagliardo–Nirenberg type inequality for difference of convex functions}
\label{sec:ineg-convex-functions}

We prove here Proposition \ref{prop:bound-diff-convex-fns}, a sort of reverse Poincaré inequality which allows to control the $\L^2$ distance $\nr{\nabla u - \nabla v}_{\L^2(K, \Rsp^d)}$ between the gradients of Lipschitz convex functions $u,v$ using the $\L^2$ distance beween these functions $\nr{u - v}_{\L^2(K)}$. This proposition is a refinement of Theorem~3.5 in \cite{chazal2010boundary}, in which the upper bound involved the uniform distance $ \nr{u - v}_{\infty}$. 
Proposition~\ref{prop:bound-diff-convex-fns} is first proven in dimension $d=1$ and on a segment (\cref{lemma:bound-diff-convex-fns}) and then generalized to higher dimensions using arguments from integral geometry.

\begin{remark}[Relation to the Gagliardo–Nirenberg inequality] 
Although the estimate of Proposition \ref{prop:bound-diff-convex-fns} resembles the Gagliardo–Nirenberg inequality, it cannot be deduced form it. More precisely, we note that without convexity of $u$ and $v$, the inequality in \eqref{eq:gn} does not hold.
One can see this by taking $u=0$ and $v_n(x) = \frac1n \sin(nx)$ on $K = [0,1]$. 
\end{remark}

\begin{remark}[Optimality of exponents]
The inequality proposed in Proposition \ref{prop:bound-diff-convex-fns} is sharp in term of the exponents of $L$ and $\nr{u - v}_{\L^2(K)}$ in the right-hand side. In the case $d=1$, let $L >0, \varepsilon > 0$ and define on $K = [0,1]$, $ u(x) =  L |x - \frac{1}{2}|$ and $v=\max(u,\eps)$.
Then $u, v$ are convex and $L$-Lipschitz and we have:
\begin{gather*}
    \nr{u - v}^2_{\L^2([0, 1])} = \frac{2}{3} \frac{\varepsilon^3}{L} \quad \text{and} \quad \nr{u' - v'}^2_{\L^2([0, 1])} = 2 L \varepsilon.
\end{gather*}
so that $\nr{u' - v'}^2_{\L^2([0, 1])} = 12^{1/3} L^{4/3} \nr{u - v}^{2/3}_{\L^2([0, 1])}.$
\end{remark}

\begin{lemma}
\label{lemma:bound-diff-convex-fns}
Let $I \subset \Rsp$ be a compact segment and let $u, v : I \to \Rsp$ be two convex functions with uniformly bounded gradients on $I$. Then
\begin{equation}
    \label{eq:stability-convex-fn-1d}
    \nr{u' - v'}_{\L^2(I)}^2 \leq 8 ( \nr{u'}_{\L^\infty(I)} + \nr{v'}_{\L^\infty(I)} )^{4/3} \nr{u - v}_{\L^2(I)}^{2/3}. 
\end{equation}
\end{lemma}
\begin{proof}
We first assume that $I = [0, 1]$. Using a simple approximation, we may assume that $u,v$ are $\Class^2$ on $I$ to get the following integration by part:
\begin{align*}
     \nr{u' - v'}_{\L^2([0, 1])}^2 &=  [(u - v)(u' - v')]_0^1 - \int_{[0, 1]} (u - v)(u'' - v''). 
\end{align*}
The convexity hypothesis then allows to get a $\L^\infty$ estimate. Indeed, 
\begin{align*}
\abs{ [(u - v)(u' - v')]_0^1 } \leq 2(\nr{u'}_{\L^\infty} + \nr{v'}_{\L^\infty}) \nr{u - v}_{\L^\infty},
\end{align*}
and by convexity
\begin{align*}
    \abs{ \int_{[0, 1]} (u - v)(u'' - v'') } &\leq \nr{u - v}_{\L^\infty} \left( \int_{[0, 1]} \abs{u''} + \int_{[0, 1]} \abs{v''} \right) \\
    &=  \nr{u - v}_{\L^\infty} \left( \int_{[0, 1]} u'' + \int_{[0, 1]} v'' \right) \\
    &\leq 2(\nr{u'}_{\L^\infty} + \nr{v'}_{\L^\infty}) \nr{u - v}_{\L^\infty}.
\end{align*}
This gives
\begin{align}
\label{eq:bound-cvx-functions-1d}
     \nr{u' - v'}_{\L^2([0, 1])}^2 &\leq 4(\nr{u'}_{\L^\infty} + \nr{v'}_{\L^\infty}) \nr{u - v}_{\L^\infty}.
\end{align}
We now  bound the $\L^\infty$ norm of $u-v$ with its $\L^2$ norm using that the Lipschitz constant of  $u-v$ is less than 
$L = \nr{u'}_{\L^\infty} + \nr{v'}_{\L^\infty}$. 
Let $\epsilon = \nr{u-v}_{\L^\infty}$ and let $x^* \in [0, 1]$ where the maximum of $\abs{u-v}$ is attained. Since $\Lip(u-v)\leq L$, one gets 
$\abs{u(x)-v(x)} \geq \frac{\eps}{2}$ on the interval $I_* = I \cap [x^* - \frac{\eps}{2L}, x^* + \frac{\epsilon}{2L}].$
The length of $I_*$ is at least $\min(\frac{\eps}{2L},1)$, so that
\begin{equation} \label{eq:compL2linf}
\nr{u - v}_{\L^2([0, 1])}^2 \geq \frac{1}{4}\min(\frac{\epsilon}{2L}, 1) \eps^2.
\end{equation}
Assume first that $\eps \leq 2L$. Then, equation~\eqref{eq:compL2linf} gives
$\eps^3 = \nr{u - v}_\infty^3 \leq 8L \nr{u - v}_{\L^2([0, 1])}^2$, thus implying

\begin{equation*}
    \nr{u - v}_{\L^\infty} \leq 2 ( \nr{u'}_{\L^\infty} + \nr{v'}_{\L^\infty} )^{1/3} \nr{u - v}_{\L^2([0, 1])}^{2/3}.
\end{equation*}
This gives, with equation \eqref{eq:bound-cvx-functions-1d}:
\begin{align}
\label{eq:stability-convex-fn-1d-[0,1]}
    \nr{u' - v'}_{\L^2([0, 1])}^2 &\leq 8 ( \nr{u'}_{\L^\infty} + \nr{v'}_{\L^\infty} )^{4/3} \nr{u - v}_{\L^2([0, 1])}^{2/3}.
\end{align}
On the other hand, if $\eps\geq 2L$, then $\nr{u - v}_{\L^2([0, 1])} \geq \frac{\eps}{2}$  by equation~\eqref{eq:compL2linf},
so that  
$$ 8 ( \nr{u'}_{\L^\infty} + \nr{v'}_{\L^\infty} )^{4/3} \nr{u - v}_{\L^2([0, 1])}^{2/3} \geq 8 L^{4/3} \left(\frac{\eps}{2}\right)^{2/3} \geq L^{4/3+2/3} = L^2,$$
which allows to conclude using $L^2 \geq \nr{u' - v'}_{\L^2([0,1])}^2$.
We get inequality \eqref{eq:stability-convex-fn-1d} for a general interval $I = [a, b]$ by an affine change of variable.

\end{proof}

The one-dimensional result from Lemma \ref{lemma:bound-diff-convex-fns} is generalized to higher dimensions thanks to two formulas from integral geometry that allow to rewrite the $\L^2$ norms of the scalar-field $u-v$ and vector-field $\nabla u - \nabla v$ over  set $K \subset \Rsp^d$ using integrals over lines intersecting  $K$.

\subsection*{Integral geometry} Denote $\Sph^{d-1}$ the unit sphere in $\Rsp^d$, and let $\sigma$ be the uniform probability measure on it. We denote $\{e\}^\perp$ the hyperplane orthogonal to a unit vector $e$. Using a simple change of variable formula, we note that for any square-integrable function $f$ on $\Rsp^d$ and any square-integrable vector field $F$ on $\Rsp^d$, one has for any $e \in \Sph^{d-1}$
\begin{align*}
    \int_{\Rsp^d} f(x)^2 \dd x &= \int_{\{e\}^\perp}\int_\Rsp f(y+te)^2 \dd t \dd \Haus^{d-1}(y),\\
    \int_{\Rsp^d} \sca{F(x)}{e}^2 \dd x &= \int_{\{e\}^\perp}\int_\Rsp \sca{F(y+te)}{e}^2 \dd t \dd \Haus^{d-1}(y),
\end{align*}
where $\Haus^k$ denotes the $k$-th dimensional Hausdorff measure.
Integrating these equalities over $e\in \Sph^{d-1}$ one gets
\begin{align*} 
\int_{\Rsp^d} f(x)^2 \dd x &= \int_{e\in\Sph^{d-1}} \left[\int_{\{e\}^\perp}\int_\Rsp  f(y+te)^2 \dd t \dd\Haus^{d-1}(y) \right]\dd \sigma(e) \\ 
\int_{\Rsp^d} \nr{F(x)}^2 \dd x &= C_d \int_{e\in\Sph^{d-1}} 
\left[\int_{\{e\}^\perp}\int_\Rsp  \sca{F(y+te)}{e}^2 \dd t \dd\Haus^{d-1}(y) \right]\dd \sigma(e).
\end{align*}
To get the second equality, in addition to the change of variable, we used  Fubini's theorem and the existence of a dimensional constant $C_d$ such that for any vector $V\in\Rsp^d$,
$$ C_d \int_{e\in \Sph^{d-1}} \sca{V}{e}^2 \dd \sigma(e) =  \nr{V}^2. $$
Finally, we will rely on  Crofton's formula -- see for instance the first paragraph of Chapter 5 in \cite{weil-hug} -- which states that for any $\mathcal{H}^{d-1}$-rectifiable subset $S$ of $\Rsp^d$ the $(d-1)$-dimensional Hausdorff measure of $S$ is proportional to the average number of intersections of a line with $S$. More precisely, there exists a dimensional constant $C'_d$ such that 
\begin{equation}
    \label{eq:crofton-d-1}
    \mathcal{H}^{d-1}(S) = C'_d \int_{e\in\Sph^{d-1}} 
    \int_{\{e\}^\perp} \#((y + \Rsp e) \cap S) \dd \Haus^{d-1}(y) \d\sigma(e),
\end{equation}
where $\# X$ is the cardinality of the set $X$. We are now ready to prove the Gagliargo-Nirenberg type inequality of Proposition \ref{prop:bound-diff-convex-fns}.

\begin{proof}[Proof of Proposition \ref{prop:bound-diff-convex-fns}]

For any $e \in \Sph^{d-1}$ and $y \in \{e\}^\perp$, denote $\ell_e^y$ the oriented line $y + e\Rsp$. Then, denoting $u_{\ell^y_e} = u|_{\ell^y_e \cap K}, v_{\ell^y_e} = v|_{\ell^y_e  \cap K}$, we know by the previous paragraph, setting $F=\nabla u - \nabla v$, that there is a dimensional constant $C_d$ such that:
\begin{equation*}
    \nr{\nabla u - \nabla v}_{\L^2(K, \Rsp^d)}^2 = C_d \int_{e\in\Sph^{d-1}} 
\int_{\{e\}^\perp}  \nr{u_{\ell^y_e}' - v_{\ell^y_e}'}^2_{\L^2(\ell^y_e \cap K)}  \dd\Haus^{d-1}(y) \dd \sigma(e).
\end{equation*}
Given any oriented line $\ell^y_e$, denote $n_{\ell^y_e} \in \Nsp\cup\{+\infty\}$ the number of connected components of $\ell^y_e\cap K$. Then, 
$n_{\ell^y_e} \leq \#(\ell^y_e \cap \partial K)$ so  that by Crofton's formula,
\begin{equation*}
    \int_{e \in \Sph^{d-1}} \int_{ \{e\}^\perp } n_{\ell^y_e} \dd\Haus^{d-1}(y) \dd \sigma(e) < +\infty.
\end{equation*}
This implies that for almost every $e \in \Sph^{d-1}$ and $y \in \{e\}^\perp$, the set $\ell^y_e \cap K$ may be decomposed as a finite union of $n_{\ell^y_e}$ segments, i.e. $\ell^y_e \cap K = \bigcup_{i = 1}^{n_{\ell^y_e}} I_{\ell^y_e}^i$. This gives
\begin{align*}
     \nr{u_{\ell^y_e}' - v_{\ell^y_e}'}_{\L^2(\ell^y_e \cap K)}^2 &= \sum_{i=1}^{n_{\ell^y_e}} \nr{u_{\ell^y_e}' - v_{\ell^y_e}'}_{\L^2(I_{\ell^y_e}^i)}^2 ,\\
     \nr{u_{\ell^y_e} - v_{\ell^y_e}}_{\L^2(\ell^y_e \cap K)}^2 &= \sum_{i=1}^{n_{\ell^y_e}} \nr{u_{\ell^y_e} - v_{\ell^y_e}}_{\L^2(I_{\ell^y_e}^i)}^2.
\end{align*}
Lemma \ref{lemma:bound-diff-convex-fns} combined with Jensen's inequality then ensure that we have for almost every $e \in \Sph^{d-1}$ and $y \in \{e\}^\perp$:
\begin{align*}
    \nr{u_{\ell^y_e}' - v_{\ell^y_e}'}_{\L^2(\ell^y_e \cap K)}^2 &\leq 8 (2 L)^{4/3} \sum_{i=1}^{n_{\ell^y_e}} \nr{u_{\ell^y_e} - v_{\ell^y_e}}_{\L^2(I_{\ell^y_e}^i)}^{2/3} \\
    &\leq  8 ( 2 L )^{4/3}  n_{\ell^y_e}^{2/3} \nr{u_{\ell^y_e} - v_{\ell^y_e}}_{\L^2(\ell^y_e \cap K)}^{2/3}.
\end{align*}
The quantity $\nr{\nabla u - \nabla v}_{\L^2(K, \Rsp^d)}^2$ is thus upper bounded by the integral
\begin{equation*}
    8 C_d (2L)^{4/3} \int_{e\in\Sph^{d-1}} 
\int_{\{e\}^\perp} n_{\ell^y_e}^{2/3} \nr{u_{\ell^y_e} - v_{\ell^y_e}}_{\L^2(\ell^y_e \cap K)}^{2/3}  \dd\Haus^{d-1}(y) \dd \sigma(e).
\end{equation*}
But Hölder's inequality together with the change of variable formula for $\nr{u - v}_{\L^2(K)}$ give
\begin{align*}
     \int_{e\in\Sph^{d-1}} &\int_{\{e\}^\perp} n_{\ell^y_e}^{2/3} \nr{u_{\ell^y_e} - v_{\ell^y_e}}_{\L^2(\ell^y_e \cap K)}^{2/3} \dd\Haus^{d-1}(y) \dd \sigma(e) \\
& \leq \left( \int_{e\in\Sph^{d-1}} \int_{\{e\}^\perp}  n_{\ell^y_e} \dd\Haus^{d-1}(y) \dd \sigma(e) \right)^{2/3} \nr{u - v}_{\L^2(K)}^{2/3}.
\end{align*}

The conclusion comes after using again that  $n_{\ell^y_e} \leq \#(\ell^y_e \cap \partial K)$ and  Crofton's formula \eqref{eq:crofton-d-1}
\begin{equation*}
    \int_{e\in\Sph^{d-1}} \int_{\{e\}^\perp}  n_{\ell^y_e} \dd\Haus^{d-1}(y) \dd \sigma(e) \leq \frac{1}{C_d'} \mathcal{H}^{d-1}(\partial K). \qedhere
\end{equation*}

\end{proof}

\subsection*{Acknowledgement} 
The authors warmly thank Bo'az Klartag, Max Fathi, Hugo Lavenant, Filippo Santambrogio, and Dorian Le Peutrec for interesting discussions related 
to this article. They acknowledge the support of the Agence national de la recherche through the project MAGA (ANR-16-CE40-0014).

\bibliographystyle{plain}
\bibliography{ref}

\end{document}